\documentclass[twoside,a4paper,12pt]{article}
\usepackage[twoside,asymmetric,bottom=4cm,bindingoffset=0pt,nomarginpar]{geometry}
\usepackage{mathtools}
\usepackage{amsmath}
\allowdisplaybreaks[4]
\usepackage{amsthm}
\usepackage{amssymb}
\usepackage{mathrsfs}
\usepackage[all]{xy}
\usepackage{tikz,tikz-cd}
\usepackage[bottom]{footmisc}
\usepackage{bm}%加粗
\usepackage{fancyhdr}%
\usepackage{geometry}
\usepackage{hyperref}
\usepackage{cite}
\usepackage{yhmath}
\usepackage{enumitem}
\usepackage{esint}
\setlist[enumerate]{label=(\arabic*)}
\newtheorem{thm}{Theorem}[section]
\newtheorem{cor}[thm]{Corollary}
\newtheorem{lem}[thm]{Lemma}
\newtheorem{prop}[thm]{Proposition}
\theoremstyle{definition}
\newtheorem{defn}[thm]{Definition}
\theoremstyle{remark}
\newtheorem{rem}[thm]{Remark}

\numberwithin{equation}{section}

\title{\textbf{Singular sets on spaces with integral curvature bounds and diffeomorphism finiteness for manifolds}}
\author{Xin Qian}
\date{}
\usepackage{fancyhdr}
\begin{document}
\maketitle
\begin{abstract}
In this paper, we are concerned with noncollapsed Riemannian manifolds $(M^{n},g)$ with integral curvature bounds, as well as their Gromov-Hausdorff limits $(M^{n}_{i},g_{i})\xrightarrow{GH}(X,d)$.\ Our main result generalizes Cheeger's Hausdorff dimension estimate for the singular set in \cite{cheeger2003integral} and improve it into Minkowski dimension estimate in the spirit of Cheeger-Naber \cite{cheeger2013lower}.\ We also prove a difffeomorphism finiteness theorem for manifolds in the critical case, using Cheeger-Naber's methods in \cite{cheeger2015regularity}.
		 
%\classification{}	
			
%\keywords{integral curvature bound, bubble tree structure, diffeomorphism finiteness}
\end{abstract}

\section{Introduction} 
Let $(M^{n},g)$ be a Riemannian manifold.\ For each $x\in M^{n}$, we denote $h(x)$ as the smallest eigenvalue for the Ricci tensor $Ric:T_{x}M\to T_{x}M$ and let $Ric_{-}(x)=\max\left\lbrace0,-h(x)\right\rbrace$.\ In this paper, we consider manifolds $(M^{n},g)$ satisfying 
 \begin{gather}
 	\int_{M^{n}}|Ric_{-}|^{p}\leqslant\lambda,\label{eq1.1}\\
 	vol(B(x,r))\geqslant vr^{n},\ \text{for all}\ x\in M^{n}\ \text{and}\ r\leqslant1,\label{eq1.2} \\
 	\int_{M^{n}}|Rm|^{q}\leqslant\Lambda,\label{eq1.3}
 \end{gather}
where $p>\frac{n}{2}$ and $1\leqslant q\leqslant\frac{n}{2}$.\ We will be particularly concerned with pointed Gromov-Hausdorff limts
\begin{align*}
	(M_{i}^{n},g_{i},x_{i})\xrightarrow{pGH}(X,d,x)
\end{align*}
of sequences of such manifolds.\ Sometimes we suppose in place of \eqref{eq1.1} that
\begin{align}
	\int_{M^{n}}|Ric|^{p}\leqslant\lambda.\label{eq1.4}
\end{align}
\indent The geometry of manifolds with bounded integral Ricci curvature has been studied extensively.\ In particular, Petersen-Wei derived a relative volume comparison theorem in \cite{petersen1997relative}.\ Afterward, under the non-collapsing condition \eqref{eq1.2}, Petersen-Wei \cite{petersen2001analysis} showed some almost rigidity results and Tian-Zhang \cite{tian2016regularity} developed a regularity theory for manifolds with bounded integral Ricci curvature.\ Also, the almost rigidity structures in collapsing case was discussed by Chen in \cite{chen2022segment}.\ Their results were directly derived from similar arguments as Cheeger-Colding \cite{cheeger1996lower,cheeger1997structure,cheeger2000structure} for manifolds with lower Ricci curvature bound.\\
\indent The main purpose of this note is to generalize Cheeger's Hausdorff dimension estimate for the singular set in \cite{cheeger2003integral}.\ Our main theorem is stated following.
\begin{thm}\label{thm1.1}
      Given $p>\frac{n}{2}$, $1\leqslant q\leqslant\frac{n}{2}$ and positive constants $v,\lambda$ and $\Lambda$, let $(M_{i}^{n},g_{i},x_{i})\xrightarrow{pGH}(X,d,x)$ be a Gromov-Hausdorff limit of manifolds satisfying \eqref{eq1.1}-\eqref{eq1.3}.\ Then the singular set $\mathcal{S}$ of $X$ satisfies
      \begin{align}
      	 \dim_{\mathcal{H}}\mathcal{S}\leqslant n-2q.\label{eq1.5}
      \end{align}
      Furthermore, the Hausdorff dimension can be improved to be the Minkowski dimension if we assume \eqref{eq1.4} instead of \eqref{eq1.1}.
\end{thm}
\indent Note that if $q>\frac{n}{2}$, then we can get a uniform lower bound on the harmonic radius (see Theorem \ref{thm2.14}) and hence, the singular set $\mathcal{S}$ must be empty.\ Indeed, the limit space $X$ is a $C^{\alpha}$-Riemannian manifold, for any $0<\alpha<2-\frac{n}{q}$.\ Thus, it is reasonable to consider $1\leqslant q\leqslant\frac{n}{2}$.\\
\indent We mention that we use Cheeger-Naber's methods in \cite{cheeger2013lower} to get volume estimates on the quantitative stratification and then derive the Minkowski dimension estimate for the singular set.\\
\indent In the critical case $q=\frac{n}{2}$, we know from Theorem \ref{thm1.1} that the singular set $\mathcal{S}$ consists of isolated points.\ Inspired by the bubble tree construction in \cite{anderson1991diffeomorphism} and \cite{cheeger2015regularity}, we prove the following diffeomorphism finiteness theorem, which generalizes Anderson-Cheeger's finiteness theorem in \cite{anderson1991diffeomorphism}.
\begin{thm}\label{thm1.2}
	Given $p>\frac{n}{2}$ and positive constants $v,D,\lambda$ and $\Lambda$, there exists $C=C(n,p,v,D,\lambda,\Lambda)$ such that the collection of closed Riemannian n-manifolds $M^{n}$ satisfying \eqref{eq1.2}, \eqref{eq1.4},
	\begin{align}
		diam(M^{n})\leqslant D,\label{eq1.6}
	\end{align}
and
\begin{align}
	\int_{M^{n}}|Rm|^{\frac{n}{2}}\leqslant\Lambda,\label{eq1.7}
\end{align}
	contains at most C diffeomorphism types.
\end{thm}
\indent Note that the all scale non-collapsing condition \eqref{eq1.2} is essential and cannot be replaced by a definite scale non-collapsing condition such as $vol(B(x,1))\geqslant v$.\ Indeed, due to an example of Yang \cite{yang1992convergence}, for any $p\geqslant\frac{n}{2}$, there exists a sequence of n-manifolds $M_{k}$ such that 
$\left\|Rm\right\|_{p}\leqslant\lambda,\ diam\leqslant D,\ vol\geqslant v$, but $b_{2}(M_{k})\to\infty$, as $k\to\infty$.\ In fact, $vol(B(x_{k},r))=o(r^{n})$ for some $x_{k}\in M_{k}$.\ In \cite{yang1992convergence, yang1992convergence2}, Yang also proved an orbifold type limit theorem (see Theorem \ref{thm2.17}), which was analogous to Anderson's result in \cite{anderson1990convergence}.\\
\indent The paper is organized as follows.\ In Section \ref{2}, we supply some known results about manifolds with integral curvature bounds.\ In Section \ref{3}, we give an $\epsilon$-regularity theorem and prove Theorem \ref{thm1.1}.\ In Section \ref{4}, we prove the finite diffeomorphism statement of Theorem \ref{thm1.2} by using the methods in \cite{cheeger2015regularity}.
\\\\
$\mathbf{\boldsymbol{\mathbf{Acknowledgement.}\mathbf{}}}$
The author would like to thank Ruobing Zhang for helpful suggestions on this research project.

\section{Preliminaries}\label{2}
In this section, we recall some notions and properties we need for manifolds with integral curvature bounds.
\subsection{Relative volume comparison}	
	Consider an $n$-manifold $M$.\ For each $x\in M$, let $h(x)$ denote the smallest eigenvalue for the Ricci tensor $Ric:T_{x}M\to T_{x}M$ and $Ric_{-}(x)=\max\left\lbrace0,-h(x)\right\rbrace$.\ Then Petersen-Wei proved the following relative volume comparison theorem in \cite{petersen1997relative}.
\begin{thm}\label{thm2.1}
	Given $p>\frac{n}{2}$, there exists $C(n,p)$ such that for any $r_{2}>r_{1}>0$,
	\begin{align}
		\left(\dfrac{vol(B(x,r_{2}))}{r_{2}^{n}}\right)^{\frac{1}{2p}}-\left(\dfrac{vol(B(x,r_{1}))}{r_{1}^{n}}\right)^{\frac{1}{2p}}\leqslant C(n,p)\left(r_{2}^{2p-n}\int_{B(x,r_{2})}|Ric_{-}|^{p}\right)^{\frac{1}{2p}}.\label{eq2.1}
	\end{align}
	Equivalently,
	\begin{align}\label{eq2.2}
		\left(\dfrac{vol(B(x,r_{1}))}{r_{1}^{n}}\right)^{\frac{1}{2p}}\geqslant\left(\dfrac{vol(B(x,r_{2}))}{r_{2}^{n}}\right)^{\frac{1}{2p}}\left(1-C(n,p)r_{2}\left(\fint_{B(x,r_{2})}|Ric_{-}|^{p}\right)^{\frac{1}{2p}}\right).
	\end{align}
\end{thm}
\begin{rem}\label{rem2.2}
By inequality \eqref{eq2.1}, if we assume $k(p,r_{2}):=\sup\limits_{x\in M}r_{2}^{2}\left(\fint_{B(x,r_{2})}|Ric_{-}|^{p}\right)^{\frac{1}{p}}$ is sufficiently small, then the volume doubling property can be easily derived (see \cite{petersen2001analysis}).\ Many works on integral Ricci curvature assume this smallness (see \cite{chen2022segment,dai2018local,petersen2001analysis} for example) and as pointed out in \cite{dai2018local}, this smallnes is necessary in order to get $L^{p}$ version of Cheeger-Colding's theory, which will be reviewed later.\ Also, under our assumption \eqref{eq1.1} and \eqref{eq1.2}, $k(p,r)\leqslant Cr^{2-\frac{n}{p}}\left\|Ric_{-}\right\|_{p}$, which is always small when the radius $r$ in consideration is small.\ This was used by Tian-Zhang in \citenum{tian2016regularity} to study almost rigidity and regularity under integral Ricci curvature bounds.
\end{rem}
Letting $r_{1}\to0$, we get the following corollary.
\begin{cor}\label{cor2.3}
	Under the asssumption \eqref{eq1.1}, we have the volume bound
	\begin{align}
		vol(B(x,r))\leqslant C(n,p)(r^{n}+\lambda r^{2p}),\ \ \forall r>0.\label{eq2.3}
	\end{align}
If we further assume \eqref{eq1.2}, then we get the volume doubling property
\begin{align}
	\frac{vol(B(x,r_{2}))}{vol(B(x,r_{1}))}\leqslant C(n,p,v,R,\lambda)\dfrac{r_{2}^{n}}{r_{1}^{n}},\ \ \forall r_{1}<r_{2}\leqslant R,\ r_{1}\leqslant1.\label{eq2.4}
\end{align}
\end{cor}

\subsection{Almost rigidity structure}	
Let $M^{n}$ be a Riemannian manifold satisfying \eqref{eq1.1} and \eqref{eq1.2} with $x\in M^{n}$.\ Denote $\omega_{n}$ to be the volume of the unit ball in $\mathbb{R}^{n}$.\ Following lines of Cheeger-Colding \cite{cheeger1996lower,cheeger1997structure, cheeger2000structure} and Colding \cite{colding1997ricci}, one may obtain the following almost rigidity properties.

\begin{thm}[Volume cnvergence,\cite{petersen2001analysis}]\label{thm2.4}
    $\forall\epsilon>0$, there exists $\delta=\delta(n,p,v,\lambda,\epsilon)$ and $r_{0}=r_{0}(n,p,v,\lambda,\epsilon)$, such that if $d_{GH}(B(x,r),B(0^{n},r))\leqslant\delta r$, for some $r\leqslant r_{0}$, then $vol(B(x,r))\geqslant(1-\epsilon)\omega_{n}r^{n}$.
\end{thm}
\begin{thm}[Almost maximal volume implies GH-close,\cite{petersen2001analysis,tian2016regularity}]\label{thm2.5}
$\forall\epsilon>0$, there exists $\delta=\delta(n,p,v,\lambda,\epsilon)$ and $r_{0}=r_{0}(n,p,v,\lambda,\epsilon)$, such that if $vol(B(x,r))\geqslant(1-\delta)\omega_{n}r^{n}$, for some $r\leqslant r_{0}$, then $d_{GH}(B(x,r),B(0^{n},r))\leqslant\epsilon r$.
\end{thm}
\begin{thm}[Almost metric cone,\cite{chen2022segment,tian2016regularity}]\label{thm2.6}
$\forall\epsilon>0$, there exists $\delta=\delta(n,p,v,\lambda,\epsilon)$ and $r_{0}=r_{0}(n,p,v,\lambda,\epsilon)$, such that if
\[\dfrac{vol(B(x,r))}{\omega_{n}r^{n}}\geqslant(1-\delta)\dfrac{vol(B(x,r/2))}{\omega_{n}(r/2)^{n}},\]
for some $r\leqslant r_{0}$, then there exists a compact length space X with $diam(X)\leqslant(1+\epsilon)\pi$ such that $d_{GH}(B(x,r),B(z^{*},r))\leqslant\epsilon r$, where $z^{*}$ is the cone vertex of $C(X)$.
\end{thm}
\indent Finally, recall that an $\epsilon$-splitting map $b=(b_{1},...,b_{k}):B(x,r)\to\mathbb{R}^{k}$ is a harmonic map such that
\begin{equation}
	\begin{gathered}
		\fint_{B(x,r)}\sum_{i,j}|\langle\nabla b_{i},\nabla b_{j}\rangle-\delta_{ij}|^{2}+r^{2}\sum_{i}|\nabla^{2}b_{i}|^{2}\leqslant\epsilon^{2},\label{eq2.5}\\
		|\nabla b_{i}|\leqslant C(n),\ \forall i.
	\end{gathered}
\end{equation}
Then we have the following.
\begin{thm}[$\epsilon$-splitting,\cite{chen2022segment,petersen2001analysis}]\label{thm2.7}
	$\forall\epsilon>0$, there exists positive $\delta=\delta(n,p,v,\lambda,\epsilon)$ and $r_{0}=r_{0}(n,p,v,\lambda,\epsilon)$ such that
	\begin{enumerate}
		\item If $d_{GH}(B(x,\delta^{-1}r),B((0^{k},y),\delta^{-1}r))\leqslant\delta r$, for some $r\leqslant r_{0}$ and $(0^{k},y)\in\mathbb{R}^{k}\times Y$, then there exists an $\epsilon$-splitting map $b:B(x,r)\to\mathbb{R}^{k}$.
		\item If there exists an $\delta$-splitting map $b:B(x,4r)\to\mathbb{R}^{k}$, for some $r\leqslant r_{0}$, then $d_{GH}(B(x,r),B((0^{k},y),r))\leqslant\epsilon r$, for some $(0^{k},y)\in\mathbb{R}^{k}\times Y$.
	\end{enumerate}
\end{thm}

\subsection{Harmonic radius and diffeomorphisms}
We first recall the notion of the harmonic radius.
\begin{defn}\label{defn2.8}
For $x\in (M^{n},g)$ and $\alpha\in(0,1)$, we define the $ C^{\alpha} $ harmonic radius $r_{h}^{\alpha}(x)$ to be the largest $r>0$ such that there exists a coordinate map $\mathcal{X}=(x^{1},...,x^{n}):B(x,r)\to\mathbb{R}^{n}$ satisfying
\begin{enumerate}
	\item $\Delta_{g}x^{i}=0$.
	\item $g_{ij}=g(\nabla x_{i},\nabla x_{j})$ on $B(x,r)$ satisfies
     \begin{align}
     	\left\|g_{ij}-\delta_{ij}\right\|_{C^{0}}+r^{\alpha}[g_{ij}]_{C^{\alpha}}\leqslant10^{-6}.\label{eq2.6}
     \end{align}
\end{enumerate}
\end{defn}
\begin{rem}\label{rem2.9}
(1)\ We call the above coordinate map $\mathcal{X}:B(x,r)\to\mathbb{R}^{n}$ a $C^{\alpha}$ harmonic coordinate system.\ Analogously, one can define the $W^{2,p}$ harmonic coordinate system and harmonic radius by requiring a scaling invariant $W^{2,p}$ bound 
		\begin{align}
			\left\|g_{ij}-\delta_{ij}\right\|_{C^{0}}+\sum_{k=1}^{2}r^{k-\frac{n}{p}}\left\|\partial^{k}g_{ij}\right\|_{L^{p}}\leqslant10^{-6},\label{eq2.7}
		\end{align}
instead of \eqref{eq2.6}.\ We will denote $r_{h}^{2,p}(x)$ to be the $W^{2,p}$ harmonic radius at $x$.\ When $p>\frac{n}{2}$ and $0<\alpha\leqslant2-\frac{n}{p}$, $ W^{2,p}$ harmonic coordinates are also $C^{\alpha}$ harmonic coordinates.\\
\indent (2)\ One may replace the bound $10^{-6}$ in \eqref{eq2.6} and \eqref{eq2.7} by any constant $Q>0$ and denote the $C^{\alpha}$ and $W^{2,p}$ harmonic radius by $r_{h}^{\alpha}(x,Q)$ and $r_{h}^{2,p}(x,Q)$ respectively.\ For manifolds satisfying \eqref{eq1.2} and \eqref{eq1.4}, we have $r_{h}^{\alpha}(x,CQ)\geqslant r_{h}^{2,p}(x,Q)$ for some constant $C=C(n,p,v,\lambda,\alpha)$ when $p>\frac{n}{2}$ and $0<\alpha\leqslant2-\frac{n}{p}$.\ However, we will omit the constants $C,Q$ and just write $r_{h}^{\alpha}(x)\geqslant r_{h}^{2,p}(x)$, since this will not affect our argument.\\
\indent (3)\ Harmonic coordinates have many good properties when it comes to regularity issues.\ We refer to \cite{anderson1993degeneration,petersen2006riemannian,petersen1997convergence} for a nice introduction.\ In particular, the $L^{p}$ bound of the Ricci curvature ($p>\frac{n}{2}$) gives \textit{a priori} $C^{\alpha}\cap W^{2,p}$ bounds on the metric tensor $g_{ij}$ in harmonic coordinates for any $\alpha\in(0,2-\frac{n}{p}]$.
\end{rem}
We now review the following $\epsilon$-regularity theorem, which can be proved by modifying arguments in \cite{anderson1990convergence}.\ See also \cite{tian2016regularity,petersen1997convergence}.
\begin{thm}\label{thm2.10}
	Given $p>\frac{n}{2}$, there exists $\epsilon=\epsilon(n,p,v,\lambda)>0$ and $r_{0}=r_{0}(n,p,v,\lambda)>0$ such that if $M^{n}$ satisfies \eqref{eq1.2}, \eqref{eq1.4} and $x\in M^{n}$ satisfies either of the following:
	\begin{enumerate}
		\item $vol(B(x,r))\geqslant(1-\epsilon)\omega_{n}r^{n},\ \text{for some}\ r\leqslant r_{0}$,
		\item $d_{GH}(B(x,r),B(0^{n},r))\leqslant\epsilon r,\ \text{for some}\ r\leqslant r_{0}$,
	\end{enumerate}
then 
\[r_{h}^{2,p}(x)\geqslant\frac{1}{2}r.\]
\end{thm}

As noted in Remark \ref{rem2.9} (2), we can get the same lower bound for $C^{\alpha}$ harmonic radius at $x$.\ Then we recall two theorems that will be used to control diffeomorphism types.\ See Theorem 8.1 and Theorem 8.4 of \cite{cheeger2015regularity}.
\begin{thm}\label{thm2.11}
	For any $\epsilon>0$, $0<\beta<\alpha<1$, there exists $\delta=\delta(n,\epsilon,\alpha,\beta)$ such that the following holds.\ If $M^{n}_{1}$ and $M^{n}_{2}$ are Riemannian manifolds and $U_{j}\subset M_{j}$ are subsets $r_{h}^{\alpha}(x)>r>0$ for all $x\in U_{j}$ and 
	\[d_{GH}(B(U_{1},r),B(U_{2},r))<\delta r,\]
	then there exists open sets $B(U_{j},\frac{r}{2})\subseteq U_{j}^{\prime}\subseteq B(U_{j},r)$ and a $C^{1,\beta}$-diffeomorphism $\phi:U_{1}^{\prime}\to U_{2}^{\prime}$ such that
	\[\left\|g_{1}-\phi^{*}g_{2}\right\|_{C^{0}}+r^{\beta}[g_{1}-\phi^{*}g_{2}]_{C^{\beta}}\leqslant\epsilon.\]
\end{thm}
\begin{thm}\label{thm2.12}
	Let $(M^{n},g)$ be a manifold satisfying \eqref{eq1.6} and \eqref{eq2.4}.\ $U\subseteq M$ is an open subset such that for some $\alpha\in(0,1)$, $r_{h}^{\alpha}(x)>rdiam(U)>0$ for all $x\in U$, where $diam(U)$ is the extrinsic diameter so that $diam(U)\leqslant D$.\ Then there exists $C=C(n,p,v,\lambda,D,r,\alpha)$ such that $U$ has at most $C$ diffeomorphism types.
\end{thm}
\begin{rem}\label{rem2.13}
	Note that by the volume doubling property \eqref{eq2.4}, $U$ may be covered by a bounded number of harmonic coordinate charts $B(x_{i},\frac{1}{2}rdiam(U))$.\ The $C^{\alpha}$ bounds on the metric in harmonic coordinate charts will yield $C^{1,\alpha}$ control over the transition maps between these charts.\ Then one can control the number of diffeomorphism types.\ See \cite{cheeger1967comparison,anderson1991diffeomorphism} for more details.
\end{rem}
\indent If we assume $L^{p}$ bound for Riemannian curvature $Rm$, then we can get the following (see Theorem 5.4 in \cite{petersen1997convergence}).
\begin{thm}\label{thm2.14}
	Given $p>\frac{n}{2},v>0$ and $\Lambda>0$, there exists $r_{0}=r_{0}(n,p,v,\Lambda)$ such that any n-manifold $(M^{n},g)$ with \eqref{eq1.2} and
	\begin{align}
		\int_{M^{n}}|Rm|^{p}\leqslant\Lambda,\label{eq2.8}
	\end{align}
satisfies $r_{h}^{2,p}(x)\geqslant r_{0}$, for all $x\in M^{n}$.\ In particular, the collection of closed n-manifolds satisfying \eqref{eq1.2}, \eqref{eq1.6} and \eqref{eq2.8} contains at most $C=C(n,p,v,D,\Lambda)$ diffeomorphism types.
\end{thm}
\begin{rem}\label{rem2.15}
As noted before, if $(M^{n}_{i},g_{i})\xrightarrow{GH}(X,d)$ with uniform lower bound on $W^{2,p}$ harmonic radius, then $X$ is a $C^{\alpha}$-Riemannian manifold and $(M^{n}_{i},g_{i})\to(X,d)$ in the $C^{\alpha}$ topology for any $\alpha\in(0,2-\frac{n}{p})$ (see \cite{petersen1997convergence}).\ Thus, it is reasonable to consider $1\leqslant q\leqslant\frac{n}{2}$ in Theorem \ref{thm1.1}.\\
\indent In addition, when $n=2$ or $3$, the Riemannian curvature is completely determined by Ricci curvature and $L^{p}$ Ricci curvature bound \eqref{eq1.4} implies $L^{p}$ bound of Riemannian curvature.\ Thus, we only need to consider the case $n\geqslant4$ when proving Theorem \ref{thm1.2}.
\end{rem}

\subsection{Structure of the limit space}
Let $(M_{i}^{n},g_{i})$ be a sequence of Riemannian $n$-manifolds satisfying \eqref{eq1.1} and \eqref{eq1.2} for some uniform $v>0,\lambda>0$ and $p>\frac{n}{2}$.\ Then by Gromov's precompactness theorem, the volume doubling property \eqref{eq2.4} implies that up to a subsequence,
\begin{align}
	(M_{i}^{n},g_{i},x_{i})\xrightarrow{pGH}(X,d,x),
\end{align}
for some complete length space $(X,d)$.\ Following lines of Cheeger-Colding \cite{cheeger1996lower,cheeger1997structure,cheeger2000structure}, one can easily get similar structure results on the limit space.\ Let us define the regular set $\mathcal{R}$ and singular set $\mathcal{S}$ on $X$ by
\begin{gather*}
	\mathcal{R}=\{x\in X:\exists\text{ tangent cone at \textit{x} is isometric to } \mathbb{R}^{n}\},\\
	\mathcal{S}=X\setminus\mathcal{R}.
\end{gather*}
Moreover, a stratification on the singular set was defined.\ Namely,
\begin{align}
	\emptyset\subset\mathcal{S}^{0}\subset...\subset\mathcal{S}^{n-1}=\mathcal{S}\subset X,
\end{align}
where
\begin{align}
	\mathcal{S}^{k}=\{x\in X:\text{ no tangent cone at \textit{x} splits off } \mathbb{R}^{k+1}\}.
\end{align}
\indent We now recall the following structure results on the limit space, which can be proved by similar arguments as that of Cheeger-Colding \cite{cheeger1996lower,cheeger1997structure,cheeger2000structure} (see also \cite{tian2016regularity} and references therein).
\begin{thm}\label{thm2.16}
	Let $(M_{i}^{n},g_{i},x_{i})\xrightarrow{pGH}(X,d,x_{\infty})$ be a sequence of manifolds satisfying \eqref{eq1.2}.Then the following hold:
	\begin{itemize}
		\item If $(M_{i}^{n},g_{i})$ satisfies \eqref{eq1.1}, then
		\begin{enumerate}
			\item[(a1)] For any $r>0$, we have: $vol(B(x_{i},r))\to\mathcal{H}^{n}(B(x_{\infty},r))$.
			\item[(a2)] For any $x\in X$, each tangent cone at $x$ is a metric cone and splits off lines isometrically.
			\item[(a3)] $\mathcal{S}=\mathcal{S}^{n-2}$, i.e., $\mathcal{S}^{n-1}\setminus\mathcal{S}^{n-2}=\emptyset$.
			\item[(a4)] $dim_{\mathcal{H}}\mathcal{S}^{k}\leqslant k$.\ In particular, $dim_{\mathcal{H}}\mathcal{S}\leqslant n-2$.
		\end{enumerate}
		\item If $(M_{i}^{n},g_{i})$ satisfies \eqref{eq1.4}, then
		\begin{enumerate}
			\item[(b1)] $\mathcal{R}$ is an open subset of $X$.
			\item[(b2)] $\mathcal{R}$ is a $C^{1,\alpha}\cap W^{3,p}$ smooth manifold with a $C^{\alpha}\cap W^{2,p}$ metric $g_{\infty}$ which induces $d$, for any $\alpha\in(0,2-\frac{n}{p})$.\ Moreover, $g_{i}$ converges to $g_{\infty}$ in the $C^{\alpha}$ and weak $W^{2,p}$ topology on $\mathcal{R}$.
		\end{enumerate}
	\end{itemize}
\end{thm}
\indent Now we review Yang's orbifold type limit theorem in \cite{yang1992convergence,yang1992convergence2}.
\begin{thm}\label{thm2.17}
	Let $(M_{i}^{n},g_{i})$ be a sequence of Riemannian manifolds satisfying \eqref{eq1.2}, \eqref{eq1.4}, \eqref{eq1.6} and \eqref{eq1.7}, for some $v,D,\lambda,\Lambda>0$ and $p>\frac{n}{2}$.\ Then there exists a subsequece converging in the Gromov-Hausdorff topology to an orbifold $(V,g)$ with finite singular points, where $g$ is a $C^{\alpha}\cap W^{2,p}$ metric off the singular points for all $\alpha\in(0,2-\frac{n}{p})$.\ Moreover, $g_{i}$ converges to $g$ in the $C^{\alpha}$ and weak $W^{2,p}$ topology away from the singular points.
\end{thm}
\begin{rem}\label{rem2.18}
	Theorem \ref{thm2.17} was analogous to Anderson's results in \cite{anderson1990convergence,anderson1989ricci} (see also \cite{bando1989construction}), where he proved a similar theorem with the stronger condition $|Ric|\leqslant\lambda$.\ In this case, one can obtain $C^{1,\alpha}\cap W^{2,q}$ regularity of the metric for every $0<\alpha<1$ and $q<\infty$, which was used in Anderson-Cheeger's proof for their diffeomorphism finiteness theorem in \cite{anderson1991diffeomorphism}.
\end{rem}
In Section \ref{4}, we will use Cheeger-Naber's argument in \cite{cheeger2015regularity}, which turns out to be more effective in nature than the one in \cite{anderson1991diffeomorphism}, to prove our diffeomorphism finiteness theorem (Theorem \ref{thm1.2}).
\subsection{Elliptic estimate}
In this subsection, we review Cheeger's elliptic estimate in \cite{cheeger2003integral}, which is crucial in proving our $\epsilon$-regularity theorem in Section \ref{3}.\\
\indent Let $[a]$ denote the greatest integer $\leqslant a$ and let $k\in\mathbb{Z}_{+}$.\ For $1\leqslant q\leqslant\frac{n}{2}$, define
\begin{align}
	\tilde{q}=
	\begin{cases}
		[q-\frac{1}{2}] & \text{if } q\neq\frac{2k+1}{2},\\
		[q-\frac{1}{2}]-1 & \text{if } q=\frac{2k+1}{2}.
	\end{cases}
\end{align}
\begin{thm}\label{thm2.19}
	Let $x\in M^{n}$ and $u:B(x,1)\to\mathbb{R}$ satisfy $\Delta u=c$ for some constant c.\ Then for $q>\frac{3}{2}$, there exists $\epsilon(q)>0$ such that 
	\begin{align}
		V^{2\tilde{q}}\fint_{B(x,\frac{3}{4})}|\nabla^{2}u|^{2q-2\tilde{q}}+V^{2q}\sum_{j=q-\tilde{q}+1}^{q}\fint_{B(x,1)}|Rm|^{j}\geqslant\epsilon(q)\fint_{B(x,\frac{1}{2})}|\nabla^{2}u|^{2q},\label{eq2.13}
	\end{align}
where $V=\sup\limits_{B(x,\frac{7}{8})}|\nabla u|$.
\end{thm}
\begin{rem}\label{rem2.20}
	Recall that if we assume $M^{n}$ satisfies \eqref{eq1.1} and \eqref{eq1.2}, then up to a rescaling on metric, $V$ can be estimated in terms of $\sup\limits_{B(x,1)}|u|$ by gradient estimates derived in \cite{dai2018local}.
\end{rem}
\begin{rem}\label{rem2.21}
	In \eqref{eq2.13}, the constant $\epsilon(q)$ goes to zero as $q\to\frac{3}{2}$.\ For $q>\frac{3}{2}$, under the condition \eqref{eq1.2} and \eqref{eq1.3}, the estimate \eqref{eq2.13} can be used to obtain $L^{2q}$ estimates for $\nabla^{2}u$, if we are given $L^{2}$ estimates for $\nabla^{2}u$.\ See Section 1 in Cheeger's paper \cite{cheeger2003integral} for more details.
\end{rem}

\section{The $\epsilon$-regularity theorem and dimension estimates for singular set}\label{3}
In this section, we present an $\epsilon$-regularity theorem which is analogous to the one in \cite{cheeger2003integral} (see also \cite{132bfa331fc94e1ea9598bc9420082e4} for the critical case $q=1$).\ And then we will apply this $\epsilon$-regularity theorem to prove Theorem \ref{thm1.1}.
\subsection{$\epsilon$-regularity theorems}
We first recall some basic facts.\ As before, we always assume $p>\frac{n}{2}$ and $1\leqslant q\leqslant\frac{n}{2}$.\\
\indent Consider a manifold $(M^{n},g)$ satisfying \eqref{eq1.1} and \eqref{eq1.2}.\ By rescaling on metric, we may assume 
\begin{gather*}
	\int_{M^{n}}|Ric_{-}|^{p}\leqslant\delta^{2p-n},\\
	vol(B(x,r))\geqslant vr^{n}\ \text{for all}\ x\in M^{n}\ \text{and}\ r\leqslant\delta^{-1}.\notag
\end{gather*}
We denote a point in $\mathbb{R}^{n-k}\times C(Y)$ by $(z,r,y)$, where $z=(z_{1},...,z_{n-k})\in\mathbb{R}^{n-k}$ and $(r,y)$ is expressed in polar coordinates of $C(Y)$.\ Let $y^{*}$ be the cone vertex of $C(Y)$.\ Assume 
\begin{align}
	d_{GH}(B(x,\delta^{-1}),B((0^{n-k},y^{*}),\delta^{-1}))<\delta.\label{eq3.1}
\end{align}
Then by Theorem \ref{thm2.7}, there is a $\psi(\delta|n,p,v)$-splitting map 
\begin{align}
b=(b_{1},...,b_{n-k}):B(x,3)\to\mathbb{R}^{n-k},\label{eq3.2}
\end{align}
where we denote $\psi(c_{1},...,c_{s}|n_{1},...,n_{t})$ to be any nonnegative function such that if $n_{1},...,n_{t}$ are fixed and $c_{1},...,c_{s}$ go to zero, then $\psi$ tends to zero.\ In addition, we know from \cite{cheeger1996lower} that there is a Gromov-Hausdorff approximation $F:B(x,\delta^{-1})\to B((0^{n-k},y^{*}),\delta^{-1})$ realizing \eqref{eq3.1} such that 
\begin{align*}
	b_{j}=z_{j}\circ F.
\end{align*}
\indent Following lines of Cheeger-Colding-Tian \cite{132bfa331fc94e1ea9598bc9420082e4}, we may obtain some additional facts.\ There exists $u:B(x,3)\to\mathbb{R}_{+}$ satisfying
\begin{gather}
	u=r\circ F,\label{eq3.3}\\
	\Delta u^{2}=2k,\notag\\
	\fint_{B(x,3)}||\nabla u^{2}|-2u|^{2}+\sum_{j}|\langle\nabla u^{2},\nabla b_{j}\rangle|<\psi(\delta|n,p,v),\label{eq3.4}\\
	\fint_{B(x,3)}|\nabla^{2}u^{2}-2(g-\sum_{i,j}db_{i}\otimes db_{j})|^{2}<\psi(\delta|n,p,v),\label{eq3.5}\\
	|\nabla u^{2}|\leqslant C(n).\notag
\end{gather}
Define $\phi:B(x,3)\to\mathbb{R}^{n-k}\times\mathbb{R}_{+}$ by 
\begin{align}
	\phi=(b,u),\label{eq3.6}
\end{align} 
and let 
\begin{align*}
	U_{r}=\phi^{-1}(B(0^{n-k},1)\times[0,r]).
\end{align*}
\indent Then there exists $C_{\delta}\subset B(0^{n-k},1)$ with $vol(C_{\delta})\geqslant(1-\psi)vol(B(0^{n-k},1))$, such that for any $z\in C_{\delta}$,
\begin{align}
	\left\lvert vol_{k}(b^{-1}(z)\cap U_{r})-\frac{r^{k}}{k}vol(Y)\right\rvert\leqslant\psi.\label{eq3.7}
\end{align}
Similarly, there exists $D_{\delta}\subset B(0^{n-k},1)\times[0,1]$ with $vol(D_{\delta})\geqslant(1-\psi)vol(B(0^{n-k},1)\times[0,1])$, such that for any $(z,r)\in D_{\delta}$,
\begin{align}
	\left\lvert vol_{k-1}(\phi^{-1}((z,r)))-r^{k-1}vol(Y)\right\rvert\leqslant\psi.\label{eq3.8}
\end{align}
\indent Before stating the $\epsilon$-regularity theorem, we first prove a lemma which is analogous to Lemma 4.14 in \cite{cheeger2003integral}.
\begin{lem}\label{lem3.1}
	Let $M^{n}$ satisfy \eqref{eq1.1} and \eqref{eq1.2}.\ Let $f:B(m,1)\to\mathbb{R}^{n-k}$ be a Lipschitz function with $|\textnormal{Lip}f|\leqslant l$ and $h:B(m,3)\to\mathbb{R}$ be a smooth function.\ Denote
	\[S_{\eta}=f\left(\{x\in B(m,1):|h(x)|\geqslant\eta\}\right).\]
	Then for all $\mu,\eta>0$, there exists $\epsilon=\epsilon(n,p,v,\lambda,l,\mu,\eta,q,k)$ such that if $2q>k$ and
	\begin{gather}
		\fint_{B(m,3)}|\nabla h|^{2q}\leqslant\mu,\\
		\fint_{B(m,3)}|h|\leqslant\epsilon,\label{eq3.10}
	\end{gather}
then $\mathcal{H}^{n-k}(S_{\eta})\leqslant\eta$.
\end{lem}
\begin{proof}
	Consider the collection 
	\[\mathcal{B}=\{B(x,r):x\in B(m,2),r\leqslant1,\fint_{B(x,r)}|\nabla h|^{2q}\geqslant Ar^{-k}\},\] where $A$ is a constant to be determined later.\ By the 5-times covering lemma, there is a disjoint subcollection $\{B(m_{t},\frac{r_{t}}{5})\}$, such that $\bigcup_{t}B(m_{t},r_{t})\supseteq\bigcup_{B\in\mathcal{B}}B$.\ Then
	\begin{align*}
	\sum_{t}\frac{vol(B(m_{t},\frac{r_{t}}{5}))}{vol(B(m,3))}\frac{A}{\mu}\left(\frac{r_{t}}{5}\right)^{-k}
		&\leqslant\sum_{t}\frac{vol(B(m_{t},\frac{r_{t}}{5}))}{vol(B(m,3))}\frac{1}{\mu}\fint_{B(m_{t},\frac{r_{t}}{5})}|\nabla h|^{2q}\\
		&\leqslant\frac{1}{\mu}\fint_{B(m,3)}|\nabla h|^{2q}\leqslant1.
	\end{align*}
	Thus, by the volume doubling property \eqref{eq2.4},
	\[\sum_{t}r_{t}^{-k}vol(B(m_{t},r_{t}))\leqslant C(n,p,v,\lambda)\frac{\mu}{A}vol(B(m,3)).\]
	By \eqref{eq1.2} and \eqref{eq2.3},
	\[\sum_{t}r_{t}^{n-k}\leqslant C(n,p,v,\lambda)\frac{\mu}{A}.\] 
	\indent We can now choose $A=A(n,p,v,\lambda,l,\mu,\eta,k)$ such that \[C(n,p,v,\lambda)\frac{\mu}{A}\leqslant\frac{\eta}{\omega_{n-k}l^{n-k}}.\]
	Since $\mathcal{H}^{n-k}(S_{\eta})\leqslant l^{n-k}\mathcal{H}^{n-k}\left(\{x\in B(m,1):|h(x)|\geqslant\eta\}\right)$, it suffices to show that there exists $\epsilon=\epsilon(n,p,v,\lambda,l,\mu,\eta,q,k)$, such that if $x\in B(m,1)$ and $|h(x)|\geqslant\eta$, then $B(x,r)\in\mathcal{B}$ for some $r\leqslant1$.\\
	\indent Fix $x\in B(m,1)$ and denote $h_{x,r}=\fint_{B(x,r)}h$.\ If $B(x,r)\notin\mathcal{B}$ for any $r\leqslant1$, then for some sufficiently small $r=r(n,p,v,\lambda,l,\mu,\eta,q,k)$,
	\begin{align}
		|h_{x,r}-h_{x,\frac{r}{2}}|\leqslant\fint_{B(x,\frac{r}{2})}|h-h_{x,r}|&\leqslant C(n,p,v,\lambda)\fint_{B(x,r)}|h-h_{x,r}|\notag\\
		&\leqslant C(n,p,v,\lambda,q)r\left(\fint_{B(x,r)}|\nabla h|^{2q}\right)^{\frac{1}{2q}}\label{eq3.11}\\
		&\leqslant C(n,p,v,\lambda,q)Ar^{1-\frac{k}{2q}}.\notag
	\end{align}
	and
	\begin{align}
	|h(x)|\leqslant|h|_{x,r}+C(n,p,v,\lambda,l,\mu,\eta,q,k)r^{1-\frac{k}{2q}}<|h|_{x,r}+\frac{\eta}{2}.\label{eq3.12}
	\end{align}
	\indent In \eqref{eq3.11}, we use the $(1,2q)$-type Neumann-Poincar\'e inequality, which can be obtained from the work of Dai-Wei-Zhang \cite{dai2018local}.\\
	\indent Finally, by \eqref{eq3.10} and volume doubling property \eqref{eq2.4}, we can choose $\epsilon$ such that $|h|_{x,r}<\frac{\eta}{2}$ and hence, $|h(x)|<\eta$.\ This completes the proof.
\end{proof}
\indent Now we can proceed to prove the $\epsilon$-regularity theorem, following the arguments of Cheeger \cite{cheeger2003integral} and Cheeger-Colding-Tian \cite{132bfa331fc94e1ea9598bc9420082e4}.
\begin{thm}\label{thm3.2}
	For any $\eta>0$, there exists $\epsilon=\epsilon(n,p,v,\lambda,\eta,q)$ and $\delta=\delta(n,p,v,\lambda,\eta,q)$, such that if $(M^{n},g)$ satisfies \eqref{eq1.1}, \eqref{eq1.2},
	\begin{align}
		\epsilon^{2q}\fint_{B(x,3\epsilon)}|Rm|^{q}\leqslant\delta,\label{eq3.13}
	\end{align}
	and for $(0^{n-k},y^{*})\in\mathbb{R}^{n-k}\times C(Y)$,
	\begin{align}
		d_{GH}(B(x,\delta^{-1}\epsilon),B((0^{n-k},y^{*}),\delta^{-1}\epsilon))<\delta\epsilon,\label{eq3.14}
	\end{align}
	where \[2q>
	\begin{cases}
		k, & \ k\ \text{even},\\
		k-1, & \ k\ \text{odd},
	\end{cases}\]
	then we have
	\[d_{GH}(B(x,\epsilon),B(0^{n},\epsilon))\leqslant\eta\epsilon.\]
\end{thm}
\begin{proof}
	Without loss of generality, let $\lambda=1$.\ Rescale the metric by a factor $\epsilon^{-2}$, we may also assume
	\begin{gather}
		\int_{M^{n}}|Ric_{-}|^{p}\leqslant\epsilon^{2p-n},\label{eq3.15}\\
		vol(B(x,r))\geqslant vr^{n}\ \text{for all}\ x\in M^{n}\ \text{and}\ r\leqslant\epsilon^{-1}.\label{eq3.16}\\
		\fint_{B(x,3)}|Rm|^{q}\leqslant\delta,\label{eq3.17}\\
		d_{GH}(B(x,\delta^{-1}),B((0^{n-k},y^{*}),\delta^{-1}))<\delta.\label{eq3.18}
	\end{gather}
	\indent We only need to show the theorem for $q>1$.\ Consider first the case in which $q>\frac{3}{2}$ and $k$ is even.\ Let $b,u,\phi$ be the maps in \eqref{eq3.2}, \eqref{eq3.3} and \eqref{eq3.6}.\ By \eqref{eq2.5}, \eqref{eq2.13} and \eqref{eq3.17}, it follows that 
	\begin{align}
		\fint_{B(x,2)}|\nabla^{2}b_{j}|^{2q}+\sum_{i,j}|\langle\nabla b_{i},\nabla b_{j}\rangle-\delta_{ij}|^{2}\leqslant\psi(\epsilon,\delta|n,p,v,q).\label{eq3.19}
	\end{align}
	For the function $u:B(x,3)\to\mathbb{R}_{+}$, we can get an $L^{2q}$ bound of $\nabla^{2}u^{2}$ from \eqref{eq2.13}, \eqref{eq3.4} and \eqref{eq3.5}.\ Then we choose $q_{1}$ satisfying $2q-1<2q_{1}<2q$.\ For any $t\in(0,\infty)$, we have the elementary inequality
	\begin{align*}
		\int f^{2q_{1}}\leqslant t^{2q_{1}-2}\int f^{2}+t^{2q_{1}-2q}\int f^{2q}.
	\end{align*} 
	Then we get 
	\begin{align*}
		\fint_{B(x,3)}|\nabla^{2}u^{2}-2(g-\sum_{i,j}db_{i}\otimes db_{j})|^{2q_{1}}<\psi.
	\end{align*}
	For the map $\phi=(b,u):B(x,3)\to\mathbb{R}^{n-k}\times\mathbb{R}_{+}$, we denote the level set by
	\begin{align*}
		U_{r}=\phi^{-1}(B(0^{n-k},1)\times[0,r]).
	\end{align*}
	By \eqref{eq2.5} and the coarea formula, there exists $(z,r)\in\mathbb{R}^{n-k}\times\mathbb{R}_{+}$ with $1-\psi\leqslant r\leqslant1$, $z$ a regular value of $b$ and $(z,r)$ a regular value of $\phi$, such that \eqref{eq3.7} holds and 
	\begin{align}
		\fint_{b^{-1}(z)\cap U_{r}}|Rm|^{q}\leqslant\psi,\label{eq3.20}\\
		\fint_{b^{-1}(z)\cap U_{r}}|\nabla^{2}b_{j}|^{2q}\leqslant\psi.\label{eq3.21}
	\end{align}
	From \eqref{eq3.19} and Lemma \ref{lem3.1} (where we need $2q>k$), with $f=b$ and $h=\langle\nabla b_{i},\nabla b_{j}\rangle-\delta_{ij}$, we can also assume 
	\begin{align}
		\sup\limits_{b^{-1}(z)\cap U_{r}}|\langle\nabla b_{i},\nabla b_{j}\rangle-\delta_{ij}|\leqslant\psi.
	\end{align}
	Similarly, we may assume \eqref{eq3.8} holds and
	\begin{gather}
		\fint_{\phi^{-1}(z,r)}|Rm|^{q}\leqslant\psi,\label{eq3.23}\\
		\fint_{\phi^{-1}(z,r)}|\nabla^{2}b_{j}|^{2q}\leqslant\psi,\label{eq3.24}\\
		\sup\limits_{\phi^{-1}(z,r)}|\langle\nabla b_{i},\nabla b_{j}\rangle-\delta_{ij}|\leqslant\psi,\label{eq3.25}\\
		\fint_{\phi^{-1}(z,r)}|\nabla^{2}u^{2}-2(g-\sum_{i,j}db_{i}\otimes db_{j})|^{2q_{1}}\leqslant\psi,\label{eq3.26}\\
		\sup\limits_{\phi^{-1}(z,r)}\left(||\nabla u^{2}|-2u|+|\langle\nabla u^{2},\nabla b_{j}\rangle|\right) \leqslant\psi.\label{eq3.27}
	\end{gather}
	Applying Gram-Schmidt process to $\nabla b_{1},...,\nabla b_{n-k},\nabla u$, we get the normal vector fields $e_{1},e_{2},...,e_{n-k},N$.\ By \eqref{eq3.21}-\eqref{eq3.27}, we obtain estimates on the second fundamental form of $b^{-1}(z)\cap U_{r}$ and $\phi^{-1}(z,r)$,
	\begin{gather}
		\fint_{b^{-1}(z)\cap U_{r}}|\nabla e_{j}|^{2q}\leqslant\psi,\label{eq3.28}\\
		\fint_{\phi^{-1}(z,r)}|\nabla e_{j}|^{2q}\leqslant\psi,\label{eq3.29}\\
		\fint_{\phi^{-1}(z,r)}|\nabla N-r^{-1}g_{\phi^{-1}(z,r)}|^{2q_{1}}\leqslant\psi,\label{eq3.30}
	\end{gather}
	where $g_{\phi^{-1}(z,r)}$ is the induced metric on $\phi^{-1}(z,r)$.\\
	\indent Note that $b^{-1}(z)\cap U_{r}$ is a $k$-dimensional submanifold with boundary $\phi^{-1}(z,r)$ and $2q>k$, $2q_{1}>k-1$.\ The Gauss equation and Chern-Gauss-Bonnet formula gives
	\begin{align*}
		\chi(b^{-1}(z)\cap U_{r})& =\int_{b^{-1}(z)\cap U_{r}}P_{\chi}+\int_{\phi^{-1}(z,r)}TP_{\chi}\\
		                         & =\pm\psi+\frac{vol(Y)}{vol(S^{k-1})}
	\end{align*}
	Since $\chi(b^{-1}(z)\cap U_{r})$ is an integer and $vol(Y)\leqslant vol(S^{k-1})+\psi$, it follows that $\chi(b^{-1}(z)\cap U_{r})=1$ and $vol(Y)=vol(S^{k-1})\pm\psi$.\ This implies $d_{GH}(B(x,1),B(0^{n},1))\leqslant\psi$, which completes the proof for $q>\frac{3}{2}$ and $k$ is even.\\
	\indent For $q>\frac{3}{2}$ and $k$ odd, we apply Lemma \ref{lem3.1} (where we need $2q>k-1$) with $f=\phi$ and $h=\langle\nabla b_{i},\nabla b_{j}\rangle-\delta_{ij}$.\ We still get estimates on $\mathrm{II}_{\phi^{-1}(z,r)}$.\ The Gauss equation and Chern-Gauss-Bonnet formula gives
	\begin{align*}
		\chi(\phi^{-1}(z,r))& =\int_{\phi^{-1}(z,r)}P_{\chi}\\
		                    & =2\frac{vol(Y)}{vol(S^{k-1})}\pm\psi.
	\end{align*}
	Since $\dim\phi^{-1}(z,r)=k-1$ is even and $\phi^{-1}(z,r)$ is a boundary, $\chi(\phi^{-1}(z,r))$ is even.\ Thus $\chi(\phi^{-1}(z,r))=2$ and $vol(Y)=vol(S^{k-1})\pm\psi$.\ As above, we complete the proof for $q>\frac{3}{2}$.\\
	\indent When $1<q\leqslant\frac{3}{2}$ and $k=2$ or 3, we do not have the elliptic estimate \eqref{eq2.13} and hence, we cannot apply Lemma \ref{lem3.1}.\ But we already have the corresponding $L^{2}$-estimates for $b$ and $u$, which is sufficient for us to obtain $L^{2}$-estimates for $\mathrm{II}_{b^{-1}(z)\cap U_{r}}$ and $\mathrm{II}_{\phi^{-1}(z,r)}$ (see Remark \ref{rem3.3} below), and follow the above proof steps.\ Thus, we complete the proof.
\end{proof}
\begin{rem}\label{rem3.3}
	The above theorem actually holds in the case $q=1$ and $k=2$ or 3.\ Note that for $k$ even, the condition $2q>k$ is only used in applying Lemma \ref{lem3.1} to control the second fundamental form.\ In fact, when $q=1$, Cheeger-Colding-Tian obtained $L^{2}$-estimates for $\mathrm{II}_{b^{-1}(z)\cap U_{r}}$ and $\mathrm{II}_{\phi^{-1}(z,r)}$ as in \eqref{eq3.28}-\eqref{eq3.30} by direct computation (see Theorem 3.7 in \cite{132bfa331fc94e1ea9598bc9420082e4}).\ We mention that they used Cheeger-Colding's cutoff function \cite{cheeger1996lower}, which can also be constructed on manifolds with integral lower Ricci bound (see \cite{dai2018local,petersen2001analysis}).\ Once we get the $L^{2}$-estimates for the second fundamental form, we can apply Gauss-Bonnet formula to $b^{-1}(z)\cap U_{r}$ when $k=2$ and to $\phi^{-1}(z,r)$ when $k=3$ as above (see also Theorem 1.11 in \cite{132bfa331fc94e1ea9598bc9420082e4}).
\end{rem}
\subsection{The Hausdorff dimension estimate}
In this subsection, we will prove the first part of Theorem \ref{thm1.1}, namely $\dim_{\mathcal{H}}\mathcal{S}\leqslant n-2q$.\ The proof follows Cheeger's arguments in \cite{cheeger2003integral}.\\
\indent We first recall some basic facts.\ The $\epsilon$-regular set $\mathcal{R}_{\epsilon}$ is defined by
\begin{align*}
	\mathcal{R}_{\epsilon}=\left\lbrace x\in X:\text{for all sufficiently small }r,\ d_{GH}(B(x,r),B(0^{n},r))\leqslant\epsilon r\right\rbrace.
\end{align*}
Let $\mathring{\mathcal{R}}_{\epsilon}$ denote the interior of $\mathcal{R}_{\epsilon}$.\ For any $\epsilon>0$, there exists $\delta>0$, such that $\mathcal{R}_{\delta}\subset\mathring{\mathcal{R}}_{\epsilon}$.\ In particular, $\mathcal{R}=\cap_{\epsilon}\mathring{\mathcal{R}}_{\epsilon}$.\\
\indent Let $(M^{n}_{i},g_{i},x_{i})\xrightarrow{pGH}(X,d,x)$ satisfy \eqref{eq1.1}-\eqref{eq1.3}.\ Then up to a subsequence, we may assume $(M^{n}_{i},g_{i},|Rm_{g_{i}}|^{q}vol_{g_{i}},x_{i})\xrightarrow{pmGH}(X,d,m_{q},x)$ in the pointed measured Gromov-Hausdorff sense, where $m_{q}$ is a Borel measure on $X$ with $m_{q}(X)\leqslant\Lambda$ (see \cite{cheeger1997structure,fukaya1987collapsing}).\ Also, from Theorem \ref{thm2.16} (a1), we know that
\begin{align}
	\fint_{B(x_{i},r)}|Rm|^{q}dvol_{g_{i}}\to\frac{m_{q}(B(x,r))}{\mathcal{H}^{n}(B(x,r))}.\label{eq3.31}
\end{align}
\begin{thm}\label{thm3.4}
	Given $p>\frac{n}{2}$, $1\leqslant q\leqslant\frac{n}{2}$ and positive constants $v,\lambda$ and $\Lambda$, let $(M_{i}^{n},g_{i},x_{i})\xrightarrow{pGH}(X,d,x)$ be a GH limit of manifolds satisfying \eqref{eq1.1}-\eqref{eq1.3}.
	\begin{enumerate}
		\item If $q$ is an integer, then $\mathcal{H}^{n-2q}(\mathcal{S}\setminus\mathcal{S}_{n-2q})=0$.
		\item If $q$ is not an integer, then $\mathcal{H}^{n-2q}(\mathcal{S})=0$.
	\end{enumerate}
	In particular, $\dim_{\mathcal{H}}\mathcal{S}\leqslant n-2q$, for all $q$.
\end{thm}
\begin{proof}
	$\forall a\geqslant0,\ \epsilon>0$, let $D_{a,\epsilon}(q)$ denote the set of points such that $y\in X\setminus\mathring{\mathcal{R}}_{\epsilon}$ and 
	\begin{align}
		\limsup_{r\to0}r^{2q}\frac{m_{q}(B(y,r))}{\mathcal{H}^{n}(B(y,r))}\leqslant a,\label{eq3.32}\\
		\lim\limits_{r\to0}r^{2q}\frac{m_{q}(B(y,r)\cap\mathring{\mathcal{R}}_{\epsilon})}{\mathcal{H}^{n}(B(y,r))}=0.\label{eq3.33}
	\end{align}
	Put	$D_{a}(q)=\bigcup_{\epsilon>0}D_{a,\epsilon}(q)$ and $D(q)=\bigcup_{a=1}^{\infty}D_{a}(q)$.\ Then by a standard covering argument and $m_{q}(X)\leqslant\Lambda$, it follows that
	\begin{gather}
		\mathcal{H}^{n-2q}(\mathcal{S}\setminus D(q))=0,\label{eq3.35}\\
		m_{q}\llcorner D(q)\ll\mathcal{H}^{n-2q}\llcorner D(q).
	\end{gather}
	Denote $m_{i,s}:=|Rm_{g_{i}}|^{s}vol_{g_{i}}$, for any $1\leqslant s\leqslant q$.\ It is clear that $m_{i,s}$ is finite on bounded sets and $m_{i,s}\ll\mathcal{H}^{n}$ on $M_{i}^{n}$.\\
	\indent Let $(M,\mathcal{H}^{n})$ be a measure space and $f$ be a nonnegative function with $\int_{M}f^{q}\leqslant\Lambda$.\ Then for any $q_{1}<q$ and $U\subset M$,
	\begin{align}
		\int_{U}f^{q_{1}}\leqslant(\mathcal{H}^{n}(U))^{1-\frac{q_{1}}{q}}\left(\int_{U}f^{q}\right)^{\frac{q_{1}}{q}}\leqslant\Lambda^{\frac{q_{1}}{q}}(\mathcal{H}^{n}(U))^{1-\frac{q_{1}}{q}}.\label{eq3.36}
	\end{align}
	Therefore, the $L^{q_{1}}$-norm of $f$ cannot concentrate on a set $U$ of small measure.\\
	\indent By passing to a subsequence in $(M^{n}_{i},g_{i},m_{i,q},x_{i})\xrightarrow{pmGH}(X,d,m_{q},x)$, we may assume $(M^{n}_{i},g_{i},m_{i,q_{1}},x_{i})\xrightarrow{pmGH}(X,d,m_{q_{1}},x)$.\ This implies that there exists a metric space $(W,d_{W})$ and isometric embeddings $\varphi_{i}:M_{i}^{n}\to W$ and $\varphi:X\to W$, such that $\varphi_{i}(x_{i})\to\varphi(x)$ in $W$ and $(\varphi_{i})_{\#}m_{i,q_{1}}\to(\varphi)_{\#}m_{q_{1}}$ in the weak* topology on $C_{bs}(W)^{*}$ (see \cite{gigli2015convergence}).\ It is obvious that $(\varphi_{i})_{\#}m_{i,q_{1}}\ll\mathcal{H}^{n}$ on $W$.\\
	\indent Then by \eqref{eq3.36}, we know that the limit measure $(\varphi)_{\#}m_{q_{1}}\ll\mathcal{H}^{n}$ on $W$ and hence, $m_{q_{1}}\ll\mathcal{H}^{n}$ on $X$.\ It follows that
	\begin{gather*}
		m_{q_{1}}(B(y,r)\cap\mathring{\mathcal{R}}_{\epsilon})=m_{q_{1}}(B(y,r)),\ \forall y\in X,\\
		D_{0}(q_{1})=D(q_{1})=\{y\in\mathcal{S}:\lim\limits_{r\to0}r^{2q_{1}}\frac{m_{q_{1}}(B(y,r))}{\mathcal{H}^{n}(B(y,r))}=0\}.
	\end{gather*}
    Notice that for any open subset $U_{i}\subset M^{n}_{i}$, we have \[r^{2q_{1}}\fint_{U_{i}}|Rm|^{q_{1}}\leqslant\left(r^{2q}\fint_{U_{i}}|Rm|^{q}\right)^{\frac{q_{1}}{q}}.\]
	Thus, we get
	\begin{gather}
		D(q)\subset D_{0}(q_{1}),\label{eq3.37}\\
		\mathcal{H}^{n-2q}(\mathcal{S}\setminus D_{0}(q_{1}))=0.\label{eq3.38}
	\end{gather}
	\indent Now we go back and prove the theorem.\ If $q$ is an integer, then we can choose $q_{1}$ such that $2q_{1}\in(2q-1,2q)$ and by Theorem \ref{thm3.2}, $\mathcal{S}\setminus\mathcal{S}^{n-2q}\subset\mathcal{S}\setminus D_{0}(q_{1})$.\ Hence, we complete the proof for (1).\\
	\indent If $q$ is not an integer, then we can choose $q_{1}$ such that $2q_{1}\in(2[q],2q)$ and choose $k=2[q]+1$ in Theorem \ref{thm3.2}.\ Thus, $\mathcal{S}\setminus\mathcal{S}^{n-k-1}\subset\mathcal{S}\setminus D_{0}(q_{1})$.\ Note that $k+1>2q$ and hence, $\mathcal{H}^{n-2q}(\mathcal{S}^{n-k-1})=0$.\ This completes the proof.
\end{proof}
The following corollary can be derived directly from Theorem \ref{thm3.4} and the results in \cite{cheeger1997structure}.
\begin{cor}\label{cor3.5}
	let $(M_{i}^{n},g_{i},x_{i})\xrightarrow{pGH}(X,d,x)$ be a GH limit of manifolds satisfying \eqref{eq1.1}-\eqref{eq1.3} with $q=\frac{n}{2}$.\ If $n$ is odd, then the singular set $\mathcal{S}$ is empty and $(X,d)$ is bi-H\"older equivalent to a Riemannian manifold.
\end{cor}
\subsection{The Minkowski dimension estimate}
In this subsection, we will prove the remaining part of Theorem \ref{thm1.1}.\ The proof relies on the Minkowski type estimates on the quantitative stratification of the singular set, introduced by Cheeger-Naber in \cite{cheeger2013lower}.
\begin{defn}\label{defn3.6}
	A metric space $Y$ is called $k$-symmetric if $Y$ is isometric to $\mathbb{R}^{k}\times C(Z)$ for some compact metric space $Z$.
\end{defn}
\begin{defn}\label{defn3.7}
	Given a metric space $Y$ with $y\in Y$, $r>0$ and $\epsilon>0$, we say that $y$ is $(k,\epsilon,r)$-symmetric if there is a $k$-symmetric metric space $X$ such that $d_{GH}(B(y,r),B(x,r))<\epsilon r$, where $x\in X$ is a vertex.
\end{defn}
\begin{defn}\label{defn3.8}
	For any $r\in(0,1)$ and $\epsilon>0$, we define the quantitative $k$-stratum of a metric space $X$ by \[\mathcal{S}^{k}_{\epsilon,r}(X):=\{x\in X: \text{for no }r\leqslant s\leqslant1\ \text{is }x\ \text{a }(k+1,\epsilon,s)\text{-symmetric point}\}.\]
\end{defn}
\indent The first main result in \cite{cheeger2013lower} is to show a Minkowski type estimate on $\mathcal{S}^{k}_{\epsilon,r}$ for noncollapsed manifolds with $Ric\geqslant-(n-1)$.\ The proof requires the cone-splitting principle and the monotonicity of the volume ratio $\frac{vol(B(x,r))}{vol_{-1}(B(r))}$.\ However, under our assumption \eqref{eq1.1} and \eqref{eq1.2}, the formula \eqref{eq2.2} tells us that the volume ratio $\frac{vol(B(x,r))}{\omega_{n}r^{n}}$ is almost monotone whenever the radius $r$ is sufficiently small.\ Then by a slight modification of Cheeger-Naber's arguments in \cite{cheeger2013lower}, one may generalize their result into the following form.
\begin{thm}\label{thm3.9}
	Let $(M^{n},g)$ satisfy \eqref{eq1.1} and \eqref{eq1.2} with $x\in M^{n}$.\ Then there exists $r_{0}=r_{0}(n,p,v,\lambda)>0$, such that for any $\epsilon,\eta>0$ and $r\leqslant r_{0}$,
	\[vol\left( T_{r}(\mathcal{S}^{k}_{\epsilon,r}(M^{n})\cap B(x,1))\right)\leqslant C(n,p,v,\lambda,\epsilon,\eta)r^{n-k-\eta}. \]
\end{thm}
The Minkowski estimates require a more rigid statement of $\epsilon$-regularity theorem.\ Namely, assuming \eqref{eq1.4} in place of \eqref{eq1.1}, we can combine Theorem \ref{thm2.10} and Theorem \ref{thm3.2} together to obtain the following. 
\begin{thm}\label{thm3.10}
	There exists $\epsilon=\epsilon(n,p,v,\lambda,q)$ and $\delta=\delta(n,p,v,\lambda,q)$, such that if $(M^{n},g)$ satisfies \eqref{eq1.2}, \eqref{eq1.4},
	\begin{align}
		\epsilon^{2q}\fint_{B(x,3\epsilon)}|Rm|^{q}\leqslant\delta,
	\end{align}
	and for $(0^{n-k},y^{*})\in\mathbb{R}^{n-k}\times C(Y)$,
	\begin{align*}
		d_{GH}(B(x,\delta^{-1}\epsilon),B((0^{n-k},y^{*}),\delta^{-1}\epsilon))<\delta\epsilon,
	\end{align*}
	where \[2q>
	\begin{cases}
		k, & \ k\ \text{even},\\
		k-1, & \ k\ \text{odd},
	\end{cases}\]
	then for any $\alpha\in(0,2-\frac{n}{p})$, we have
	\[r^{\alpha}_{h}(x)>\frac{1}{2}\epsilon.\]
\end{thm}
From the proof of Theorem \ref{thm3.2}, we may fix $\delta=\delta(n,p,v,\lambda,q)<1$ with the property that there exists $\epsilon_{0}=\epsilon_{0}(n,p,v,\lambda,q)$ such that the above theorem holds for every $0<\epsilon\leqslant\epsilon_{0}$.\ Without loss of generality, we may also assume $\delta^{-1}\epsilon_{0}<r_{0}$, where $r_{0}=r_{0}(n,p,v,\lambda)$ is the constant in Theorem \ref{thm3.9}.\\
\indent Now we combine Theorem \ref{thm3.9} and Theorem \ref{eq3.10} to prove the following theorem, which implies the Minkowski dimension estimate of Theorem \ref{thm1.1}.
\begin{thm}\label{thm3.11}
	Let $(M^{n},g)$ satisfy \eqref{eq1.2}-\eqref{eq1.4} with $x\in M^{n}$.\ Then there exists $\epsilon_{0}=\epsilon_{0}(n,p,v,\lambda,q)>0$, such that for any $\eta>0$, $\alpha\in(0,2-\frac{n}{p})$ and $\epsilon\leqslant\epsilon_{0}$,
	\begin{align}
		vol\left( T_{\epsilon}( \{y\in B(x,1):r^{\alpha}_{h}(y)\leqslant\frac{\epsilon}{2}\})\right)\leqslant C(n,p,v,\lambda,\Lambda,q,\eta)\epsilon^{2q-\eta}.\label{eq3.40}
	\end{align}
\end{thm}
\begin{proof}
	As noted above, let us fix $\delta=\delta(n,p,v,\lambda,q)<1$ and $\epsilon_{0}=\epsilon_{0}(n,p,v,\lambda,q)$.\ Consider first the case in which $q$ is an integer.\\
	\indent For arbitrary $\epsilon\leqslant\epsilon_{0}$, the above $\epsilon$-regularity theorem (Theorem \ref{thm3.10}) gives the following inclusion
	\[\left\lbrace y\in B(x,1):r^{\alpha}_{h}(y)\leqslant\frac{1}{2}\epsilon\right\rbrace \subseteq\left( \mathcal{S}^{n-2q}_{\delta^{2},\delta^{-1}\epsilon}\cup\{y:\epsilon^{2q}\fint_{B(y,3\epsilon)}|Rm|^{q}>\delta\}\right)\cap B(x,1).\]
	By Theorem \ref{thm3.9}, we know that
	\begin{align}
		vol\left( T_{\epsilon}(\mathcal{S}^{n-2q}_{\delta^{2},\delta^{-1}\epsilon}\cap B(x,1))\right)\notag
		&\leqslant vol\left( T_{\delta^{-1}\epsilon}(\mathcal{S}^{n-2q}_{\delta^{2},\delta^{-1}\epsilon}\cap B(x,1))\right)\\
		&\leqslant C(n,p,v,\lambda,\delta,\eta)\epsilon^{2q-\eta}\label{eq3.41}\\
		&\leqslant C(n,p,v,\lambda,q,\eta)\epsilon^{2q-\eta}\notag.
	\end{align}
	Then we denote $A=\{y\in B(x,1):\epsilon^{2q}\fint_{B(y,3\epsilon)}|Rm|^{q}>\delta\}$.\\
	\indent Now, our goal is to get volume estimate of the $\epsilon$-tubular neighborhood $T_{\epsilon}(A)$.\ Consider the collection
	\begin{align*}
		\mathcal{B}=\left\lbrace B(y,3\epsilon): y\in A\right\rbrace. 
	\end{align*} 
	By the 5-times covering lemma, there exists a pairwise disjoint subcollection $\{B(y_{i},3\epsilon)\}$ such that 
	\begin{align}
		\bigcup_{B\in\mathcal{B}}B\subset\bigcup_{i}B(y_{i},15\epsilon).
	\end{align}
	Thus, we get
	\begin{align}
		\sum_{i}vol(B(y_{i},3\epsilon))<\frac{\epsilon^{2q}}{\delta}\sum_{i}\int_{B(y_{i},3\epsilon)}|Rm|^{q}\leqslant C(\delta,\Lambda)\epsilon^{2q},
	\end{align}
	and by volume doubling property \eqref{eq2.4},
	\begin{align}\label{eq3.44}
		vol(T_{\epsilon}(A))\leqslant\sum_{i}vol(B(y_{i},15\epsilon))\leqslant C(n,p,v,\lambda,\Lambda,q)\epsilon^{2q}.
	\end{align}
	Combining \eqref{eq3.41} and \eqref{eq3.44}, we obtain the desired estimate \eqref{eq3.40}.\\
	\indent Now, we consider the case in which $q$ is not an integer.\ Similarly, by Theorem \ref{thm3.10}, we have the inclusion
	 \[\left\lbrace y\in B(x,1):r^{\alpha}_{h}(y)\leqslant\frac{1}{2}\epsilon\right\rbrace \subseteq\left( \mathcal{S}^{n-2[q]-2}_{\delta^{2},\delta^{-1}\epsilon}\cup\{y:\epsilon^{2q}\fint_{B(y,3\epsilon)}|Rm|^{q}>\delta\}\right)\cap B(x,1).\]
	 Following the same argument as above, we can obtain a stronger estimate:
	 \begin{align*}
	 	vol\left( T_{\epsilon}( \{y\in B(x,1):r^{\alpha}_{h}(y)\leqslant\frac{\epsilon}{2}\})\right)\leqslant C(n,p,v,\lambda,\Lambda,q)\epsilon^{2q}.
	 \end{align*}
	 Thus, we complete the proof.
\end{proof}
The Minkowski dimension estimate of Theorem \ref{thm1.1} can be viewed as an easy corollary of Theorem \ref{thm3.11} (see similar arguments in Section 7 of \cite{cheeger2015regularity}).\ Thus, we complete the proof of Theorem \ref{thm1.1}.

\section{Diffeomorphism finiteness}\label{4}
In this section, we will prove our diffeomorphism finiteness theorem (Theorem \ref{thm1.2}).\ The idea is following Cheeger-Naber's bubble tree construction in Section 8 of \cite{cheeger2015regularity}.
\subsection{Neck regions}
In this subsection, we prove some technical results to handle the neck regions in the bubble tree structure.\\
\indent From formula \eqref{eq2.2}, under the assumption \eqref{eq1.2} and \eqref{eq1.4}, the volume ratio $\frac{vol(B(x,r))}{r^{n}}$ will be almost monotone whenever the radius $r$ is sufficiently small.\ For convenience in writing, we use the notation
\begin{align}
\mathcal{V}_{r}(x):=-\ln\left(\frac{vol(B(x,r))}{\omega_{n}r^{n}}\right)
\end{align}
where $\omega_{n}$ is the volume of unit ball in $\mathbb{R}^{n}$.\\
\indent Let $A_{r_{1},r_{2}}(q)$ denote the geodesic annulus centered at some $q\in M^{n}$, i.e.,
\[A_{r_{1},r_{2}}(q):=B(q,r_{2})\setminus\overline{B(q,r_{1})}.\]We begin with the following annulus estimate.\ As before, we always assume $n\geqslant4,\ p>\frac{n}{2}$ and $\alpha\in(0,2-\frac{n}{p})$.\ In fact, one may set $\alpha$ as a fixed value, such as $1-\frac{n}{2p}$.
\begin{prop}\label{prop4.1}
	For any $\epsilon\in(0,\frac{1}{2})$, there exists $\delta=\delta(n,p,v,\Lambda,\epsilon)$ such that if a manifold $(M^{n},g)$ satisfies 
	\begin{gather}
		\int_{M^{n}}|Rm|^{\frac{n}{2}}\leqslant\Lambda,\label{eq4.2}\\
		\int_{M^{n}}|Ric|^{p}\leqslant\delta^{2p-n},\label{eq4.3}\\
		vol(B(x,r))\geqslant vr^{n}\ \text{for all}\ x\in M^{n}\ \text{and}\ r\leqslant\delta^{-1},\label{eq4.4}\\
		|\mathcal{V}_{4}(q)-\mathcal{V}_{2}(q)|\leqslant\delta,\label{eq4.5}
	\end{gather}
then there exists $r_{0}=r_{0}(n,v)$ and a discrete subgroup $\Gamma\subseteq O(n)$ with $|\Gamma|<C(n,v)$ such that the following hold:
\begin{enumerate}
	\item for any $x\in A_{\epsilon,2}(q)$, $r_{h}^{\alpha}(x)>r_{0}\epsilon$;
	\item there exists a subset $A_{\epsilon,2}(q)\subseteq U\subseteq A_{\frac{1}{2}\epsilon,2+\frac{1}{2}\epsilon}(q)$ and a diffeomorphism $\phi:A_{\epsilon,2}(0)\to U$ where $0\in\mathbb{R}^{n}/\Gamma$, such that 
	\[\left\|\phi^{*}g_{ij}-\delta_{ij}\right\|_{C^{0}}+[\phi^{*}g_{ij}]_{C^{\alpha}}\leqslant\epsilon.\]
\end{enumerate}
\end{prop}
\begin{proof}
Argue by contradiction.\ We assume that there exists a sequence of manifolds $(M^{n}_{i},g_{i},q_{i})$ with $\int_{M_{i}^{n}}|Rm|^{\frac{n}{2}}\leqslant\Lambda,\ \int_{M_{i}^{n}}|Ric|^{p}\leqslant\delta_{i}^{2p-n}\to0,\ vol(B(x_{i},r))\geqslant vr^{n}\ \text{for all}\ x_{i}\in M_{i}^{n}, r\leqslant\delta_{i}^{-1}$ and $|\mathcal{V}_{4}(q_{i})-\mathcal{V}_{2}(q_{i})|\leqslant\delta_{i}\to0$, but the conclusions fail.\ After passing to a subsequence, we get 
\[(M^{n}_{i},g_{i},q_{i})\xrightarrow{GH}(X,d,q).\]
By Theorem \ref{thm2.6}, we know that $B(q,4)=B(z^{*},4)$ where $z^{*}$ is the cone vertex of a metric cone $C(Z)$ with $diam(Z)\leqslant\pi$.\ Choose a sequence $r_{i}\to0$ and let $\tilde{g}_{i}=r_{i}^{-2}g_{i}$.\ Then passing to a subsequence, we may assume
\[(M^{n}_{i},\tilde{g}_{i},q_{i})\xrightarrow{GH}(C(Z),d_{C},z^{*}).\]
\indent Using Theorem \ref{thm1.1} and the scaling invariance of metric cone, we know that $z^{*}$ is the only possible singularity of $C(Z)$ (otherwise, there is a ray full of singularities) and by Theorem \ref{thm2.16} (b2), $\tilde{g}_{i}$ converges to $g_{C}$ in the $C^{\alpha}$ and weak $W^{2,p}$ topology away from the vertex $z^{*}$.\ Considering the equation for Ricci curvature in harmonic coordinates and using $\left\|Ric_{\tilde{g}_{i}}\right\|_{L^{p}}\to0$, it follows that $(C(Z),g_{C})$ is smooth and Ricci flat away from the vertex $z^{*}$.\ Also, $\tilde{g}_{i}$ converges to $g_{C}$ in the strong $W^{2,p}$ topology away from $z^{*}$.\ Since $\int_{M_{i}^{n}}|Rm_{\tilde{g}_{i}}|^{\frac{n}{2}}\leqslant\Lambda$, we obtain that 
\begin{align}
	\int_{C(Z)\setminus\{z^{*}\}}|Rm_{g_{C}}|^{\frac{n}{2}}\leqslant\Lambda.\label{eq4.6}
\end{align}
However, by scaling invariance of the metric cone $C(Z)$ and the curvature energy $\left\|Rm\right\|_{L^{\frac{n}{2}}}$, it follows that for any $r>0$,
\begin{align}
  \int_{A_{r,2r}(z^{*})}|Rm_{g_{C}}|^{\frac{n}{2}}=\int_{A_{2r,4r}(z^{*})}|Rm_{g_{C}}|^{\frac{n}{2}}.\label{eq4.7}
\end{align}
Combining \eqref{eq4.6} and \eqref{eq4.7}, $\int_{C(Z)\setminus\{z^{*}\}}|Rm_{g_{C}}|^{\frac{n}{2}}=0$ and $C(Z)$ is a flat cone.\ Thus $Z=S^{n-1}/\Gamma$ has constant sectional curvature +1.\ In addition, we know from Theorem \ref{thm2.16} (a1) that $|\Gamma|<C(n,v)$.\\
\indent Now we come back to the original converging sequence $(M^{n}_{i},g_{i},q_{i})\xrightarrow{GH}(X,d,q)$.\ We have already known that $B(q,4)=B(0,4)$ where $0\in\mathbb{R}^{n}/\Gamma$.\ It follows that there exists $r_{0}=r_{0}(n,v)$ such that for $z\in\mathbb{R}^{n}/\Gamma$ with $|z|=1$, we have $B(z,2r_{0})=B(0^{n},2r_{0})$ where $0^{n}\in\mathbb{R}^{n}$.\ In particular, for sufficiently large $i$, we obtain from Theorem \ref{thm2.10} that $r^{\alpha}_{h}(x)\geqslant r^{2,p}_{h}(x)>r_{0}\epsilon$, for all $x\in A_{\epsilon,2}(q_{i})$ and $\alpha\in(0,2-\frac{n}{p})$.\ Thus, (1) is true and (2) must fail to hold.\\
\indent However, by (1) and Theorem \ref{thm2.11}, we have that for sufficiently large $i$, there exists a diffeomorphism $\phi_{i}:A_{\epsilon,2}(0)\to M_{i}^{n}$ where $0\in\mathbb{R}^{n}/\Gamma$, such that the pullback metric $\phi_{i}^{*}g_{i}$ is $\epsilon$-close to $\delta_{ij}$ in $C^{\alpha}$ norm.\ This implies (2); a contradiction.
\end{proof}
\begin{rem}\label{rem4.2}
	Note that since we have taken scaling in \eqref{eq4.3} and \eqref{eq4.4}, we do not need to consider small radius $r$ as in Theorem \ref{thm2.6} and Theorem \ref{thm2.10}.
\end{rem}
\indent Motivated by Proposition \ref{prop4.1}, we introduce the following definition and lemma similar to those in \cite{cheeger2015regularity} before stating our key neck theorem.
\begin{defn}\label{defn4.3}
	Consider the scales $r_{k}=2^{-k}$, $k\in\mathbb{N}$.\ For any $x\in M^{n}$, we define
	\[T_{k}^{\delta}(x):=
	\begin{cases}
		1 & \text{if } |\mathcal{V}_{4r_{k}}(x)-\mathcal{V}_{2r_{k}}(x)|>\delta,\\
		0 & \text{if } |\mathcal{V}_{4r_{k}}(x)-\mathcal{V}_{2r_{k}}(x)|\leqslant\delta.
	\end{cases}\]
We also denote by $|T^{\delta}|(x)=\sum_{k}T_{k}^{\delta}(x)$ the number of bad scales at $x$.
\end{defn}
\begin{lem}\label{lem4.4}
	There exists $\delta=\delta(n,p,v)$ satisfying the following property.\ Let $(M^{n},g)$ be a manifold with $\int_{M^{n}}|Ric_{-}|^{p}\leqslant\delta^{2p-n}$ and $vol(B(x,r))\geqslant vr^{n}\ \text{for all}\ x\in M^{n}\ \text{and}\ r\leqslant\delta^{-1}$.\ Then for any $\delta^{\prime}>0$ and $x\in M^{n}$, there exists at most $N=N(n,p,v,\delta^{\prime})$ scales $k\in\mathbb{N}$ such that $T_{k}^{\delta^{\prime}}(x)=1$.
\end{lem}
\begin{proof}
	For any $x\in M^{n}$, $\mathcal{V}_{1}(x)\leqslant-\ln(\frac{v}{\omega_{n}})=C(n,v)$.\ And the almost monotonicity formula \eqref{eq2.2} gives: $\forall r<R\leqslant1$,
	\[\mathcal{V}_{R}(x)-\mathcal{V}_{r}(x)\geqslant2p\ln(1-C(n,p,v)(\delta R)^{1-\frac{n}{2p}})\geqslant2p\ln(1-\frac{1}{2}R^{1-\frac{n}{2p}}).\]
	Here, we have chosen $\delta=(2C(n,p,v))^{\frac{2p}{n-2p}}$.\ Then
	\begin{align*}
	  C(n,v) & \geqslant\mathcal{V}_{1}(x)-\mathcal{V}_{0}(x)\\
	  & =\sum_{k}(\mathcal{V}_{r_{k}}(x)-\mathcal{V}_{r_{k+1}}(x))\\
	  & \geqslant\sum_{k}|\mathcal{V}_{r_{k}}(x)-\mathcal{V}_{r_{k+1}}(x)|+4p\sum_{k}\ln(1-\frac{1}{2}r_{k}^{1-\frac{n}{2p}})\\
	  & \geqslant\sum_{k}|\mathcal{V}_{r_{k}}(x)-\mathcal{V}_{r_{k+1}}(x)|-8p\sum_{k}2^{-k(1-\frac{n}{2p})}\\
	  & \geqslant\sum_{k}|\mathcal{V}_{r_{k}}(x)-\mathcal{V}_{r_{k+1}}(x)|-C(n,p).
	\end{align*}
Then $\sum|\mathcal{V}_{r_{k}}(x)-\mathcal{V}_{r_{k+1}}(x)|\leqslant C(n,p,v)$.\ In particular, there are at most $N=C(n,p,v)(\delta^{\prime})^{-1}$ scales $k\in\mathbb{N}$ such that $|\mathcal{V}_{r_{k}}(x)-\mathcal{V}_{r_{k+1}}(x)|>\delta^{\prime}$.
\end{proof}
\indent Note that for any good scale $k\in\mathbb{N}$ at $q$, i.e., $T_{k}^{\delta}(q)=0$, we can denote by $\Gamma_{k}\subseteq O(n)$ the subgroup arising from Proposition \ref{prop4.1}.\ The following is our key neck theorem, which is similar to Proposition \ref{prop4.1} while $\delta$ is independent of the ratio $r_{1}/r_{2}$ of the annulus $A_{r_{1},r_{2}}(q)$.\ Generally, we are aimed to construct a bubble tree decomposition to prove the diffeomorphism finiteness and in the decomposition, we will encounter two different types of regions: body regions and neck regions.\ The  neck regions will be diffeomorphic to cylinders $(0,1)\times S^{n-1}/\Gamma$ and connect different body regions.
\begin{thm}[Neck Theorem]\label{thm4.5}
	For any $0<\epsilon\leqslant\epsilon(n,v)$, there exists $\delta=\delta(n,p,v,\Lambda,\epsilon)$ with the following property.\ Let $(M^{n},g)$ satisfy 
	\begin{gather*}
		\int_{M^{n}}|Rm|^{\frac{n}{2}}\leqslant\Lambda,\\
		\int_{M^{n}}|Ric|^{p}\leqslant\delta^{2p-n},\\
		vol(B(x,r))\geqslant vr^{n}\ \text{for all}\ x\in M^{n}\ \text{and}\ r\leqslant\delta^{-1},
	\end{gather*}
and assume $k_{1}\in\mathbb{N}$ satisfies $T_{k_{1}}^{\delta}(q)=0$ with $\Gamma_{k_{1}}$ the corresponding group.\ If $k_{2}\geqslant k_{1}$ satisfies $\mathcal{V}_{r_{k_{2}}}(q)\geqslant\ln|\Gamma_{k_{1}}|-\delta^{\frac{1}{2}(1-\frac{n}{2p})}$, then there exists a subset $A_{\frac{1}{2}r_{k_{2}},2r_{k_{1}}}(q)\subseteq U\subseteq A_{\frac{1}{2}(1-\epsilon)r_{k_{2}},2(1+\epsilon)r_{k_{1}}}(q)$ and a diffeomorphism $\phi:A_{\frac{1}{2}r_{k_{2}},2r_{k_{1}}}(0)\to U$, where $0\in\mathbb{R}^{n}/\Gamma_{k_{1}}$, such that
\[\left\|\phi^{*}g_{ij}-\delta_{ij}\right\|_{C^{0}(A_{\frac{1}{2}r_{k_{2}},2r_{k_{1}}}(0))}+r_{k_{2}}^{\alpha}[\phi^{*}g_{ij}]_{C^{\alpha}(A_{\frac{1}{2}r_{k_{2}},2r_{k_{1}}}(0))}\leqslant\epsilon.\]
\end{thm}
\begin{proof}
	Let $\epsilon_{1}>0$ be arbitrary and $\delta(n,p,v,\Lambda,\epsilon_{1})$ be the corresponding number from Proposition \ref{prop4.1}.\ This number $\delta$ will be fixed later to be the number we need.\ If $T_{k_{1}}^{\delta}(q)=0$, then there exists a diffeomorphism $\phi_{k_{1}}:A_{\frac{1}{2}r_{k_{1}},2r_{k_{1}}}(0)\to U_{k_{1}}$, where $0\in\mathbb{R}^{n}/\Gamma_{k_{1}}$ and $A_{\frac{1}{2}r_{k_{1}},2r_{k_{1}}}(q)\subseteq U_{k_{1}}\subseteq A_{\frac{1}{2}(1-\epsilon_{1})r_{k_{1}},2(1+\epsilon_{1})r_{k_{1}}}(q)$, such that
	\[\left\|\phi_{k_{1}}^{*}g_{ij}-\delta_{ij}\right\|_{C^{0}}+r_{k_{1}}^{\alpha}[\phi_{k_{1}}^{*}g_{ij}]_{C^{\alpha}}\leqslant\epsilon_{1}.\]
	Moreover, let $\epsilon>0$ be fixed and $4\delta_{1}(n,p,v,\Lambda,\epsilon)$ be the corresponding number from Proposition \ref{prop4.1}.\ Then we can choose $\epsilon_{1}=\epsilon_{1}(n,p,v,\Lambda,\delta_{1})=\epsilon_{1}(n,p,v,\Lambda,\epsilon)$ with $\delta=\delta(n,p,v,\Lambda,\epsilon_{1})=\delta(n,p,v,\Lambda,\epsilon)$ sufficiently small such that $\delta^{\frac{1}{2}(1-\frac{n}{2p})}\leqslant\delta_{1}$ and
	\[\mathcal{V}_{2r_{k_{1}}}(q)<\ln|\Gamma_{k_{1}}|+\delta_{1}.\]
	Thus, if $k_{2}\geqslant k_{1}$ satisfies $\mathcal{V}_{r_{k_{2}}}(q)\geqslant\ln|\Gamma_{k_{1}}|-\delta^{\frac{1}{2}(1-\frac{n}{2p})}\geqslant\ln|\Gamma_{k_{1}}|-\delta_{1}$, then by the almost monotonicity formula \eqref{eq2.2}, we have $T_{k}^{4\delta_{1}}(q)=0$ for all scales $k\in[k_{1},k_{2}]\cap\mathbb{N}$.\\
	\indent By Proposition \ref{prop4.1}, for each $k\in[k_{1},k_{2}]$, there exists a diffeomorphism $\phi_{k}:A_{\frac{1}{2}r_{k},2r_{k}}(0)\to U_{k}$, where $0\in\mathbb{R}^{n}/\Gamma_{k}$ and $A_{\frac{1}{2}r_{k},2r_{k}}(q)\subseteq U_{k}\subseteq A_{\frac{1}{2}(1-\epsilon)r_{k},2(1+\epsilon)r_{k}}(q)$, such that
	\begin{align}\label{eq4.8}
		\left\|\phi_{k}^{*}g_{ij}-\delta_{ij}\right\|_{C^{0}(A_{\frac{1}{2}r_{k},2r_{k}}(0))}+r_{k}^{\alpha}[\phi_{k}^{*}g_{ij}]_{C^{\alpha}(A_{\frac{1}{2}r_{k},2r_{k}}(0))}\leqslant\epsilon.
	\end{align}
	In particular, this implies that $\Gamma_{k}=\Gamma$ is independent of $k$.\\
	\indent Then we consider the inverse maps $\phi_{k}^{-1}:U_{k}\to A_{\frac{1}{2}r_{k},2r_{k}}(0)$.\ Observe that after composing $\phi_{k}$ with a rotation of $\mathbb{R}^{n}/\Gamma$, we can assume that 
	\begin{align}\label{eq4.9}
		|\phi_{k}^{-1}(x)-\phi_{l}^{-1}(x)|\leqslant\epsilon r_{k},\ \forall x\in U_{k}\cap U_{l}.
	\end{align}
	Now we can choose a sufficiently small $\epsilon(n,v)$, such that for any $0<\epsilon\leqslant\epsilon(n,v)$, if $y\in\mathbb{R}^{n}/\Gamma$, then $B(y,4\epsilon|y|)=B(0^{n},4\epsilon|y|)$ where $0^{n}\in\mathbb{R}^{n}$.\\
	\indent For each $k$, let $\eta_{k}:U_{k}\to\mathbb{R}$ be a smooth cutoff function with
	\[\eta_{k}(x):=
	\begin{cases}
		1 & \text{if } x\in A_{\frac{5}{8}r_{k},\frac{15}{8}r_{k}}(q),\\
		0 & \text{if } x\notin A_{\frac{1}{2}r_{k},2r_{k}}(q)
	\end{cases}\]
and $|\nabla\eta_{k}|<10r_{k}^{-1}$.\ If we let $\eta(x)=\sum_{k}\eta_{k}(x)$, then
\[\psi_{k}:=\frac{\eta_{k}}{\eta}:U_{k}\to\mathbb{R}\]
satisfies $\sum_{k}\psi_{k}(x)=1$ and $|\nabla\psi_{k}|<200r_{k}^{-1}$.\ Now we define the map 
\[\phi^{-1}:U=\bigcup_{k}U_{k}\to A_{\frac{1}{2}r_{k_{2}},2r_{k_{1}}}(0)\] by $\phi^{-1}(x)=\sum_{k}\psi_{k}(x)\phi_{k}^{-1}(x)$.\ Notice that all the $\phi_{k}^{-1}(x)$ live in a Euclidean ball, so the convex combination is well defined.\ On each domain $U_{k}$, $\phi^{-1}$ and $\phi_{k}^{-1}$ are $C^{1}$-close by \eqref{eq4.8} and \eqref{eq4.9}.\ Thus, $\phi^{-1}:U\to A_{\frac{1}{2}r_{k_{2}},2r_{k_{1}}}(0)$ is a diffeomorphism.\ Again by \eqref{eq4.8} and \eqref{eq4.9}, we can derive the desired estimate:
\[\left\|\phi^{*}g_{ij}-\delta_{ij}\right\|_{C^{0}(A_{\frac{1}{2}r_{k_{2}},2r_{k_{1}}}(0))}+r_{k_{2}}^{\alpha}[\phi^{*}g_{ij}]_{C^{\alpha}(A_{\frac{1}{2}r_{k_{2}},2r_{k_{1}}}(0))}\leqslant400\epsilon.\]
This completes the proof.
\end{proof}
\indent To end this section, we prove the following lemma which will be used to control the number of the neck regions and body regions.
\begin{lem}\label{lem4.6}
	For any sufficiently small $\delta\leqslant\delta(n,p,v)$, there exists a scale $k=k(n,p,v,\Lambda,\delta)$ with the following property.\ Let $M^{n}$ be a manifold with $\int_{M^{n}}|Rm|^{\frac{n}{2}}\leqslant\Lambda$, $\int_{M^{n}}|Ric|^{p}\leqslant\delta^{2p-n}$, $vol(B(x,r))\geqslant vr^{n}\ \text{for all}\ x\in M^{n},\ r\leqslant\delta^{-1}$ and $\mathcal{V}_{1}(q)<\ln N-\delta^{\frac{1}{2}(1-\frac{n}{2p})}$ for some integer $N\geqslant2$.\ Then
	\[\mathcal{V}_{r_{k}}(q)<\ln(N-1)+\delta^{\frac{1}{2}(1-\frac{n}{2p})}.\]
\end{lem}
\begin{proof}
	Let $\delta\leqslant\delta(n,p,v)$ satisfy Lemma \ref{lem4.4}.\ By Proposition \ref{prop4.1}, we can choose $\delta^{\prime}(n,p,v,\Lambda,\delta)$ such that if $\int_{M^{n}}|Rm|^{\frac{n}{2}}\leqslant\Lambda$, $\int_{M^{n}}|Ric|^{p}\leqslant(\delta^{\prime})^{2p-n}$, $vol(B(x,r))\geqslant vr^{n}\ \text{for all}\ x\in M^{n}, r\leqslant(\delta^{\prime})^{-1}$ and $T_{1}^{\delta^{\prime}}(q)=0$, then \[|\mathcal{V}_{1}(q)-\ln|\Gamma||<\frac{1}{2}\delta^{\frac{1}{2}(1-\frac{n}{2p})}.\]
	By rescaling, we see that if $\int_{M^{n}}|Rm|^{\frac{n}{2}}\leqslant\Lambda$, $\int_{M^{n}}|Ric|^{p}\leqslant\delta^{2p-n}$, $vol(B(x,r))\geqslant vr^{n}\ \text{for all}\ x\in M^{n}, r\leqslant\delta^{-1}$ and $|\mathcal{V}_{2r_{k}}(q)-\mathcal{V}_{r_{k}}(q)|\leqslant\delta^{\prime}$ with $k\geqslant k(n,p,v,\Lambda,\delta)$, then \[|\mathcal{V}_{r_{k}}(q)-\ln|\Gamma||<\frac{1}{2}\delta^{\frac{1}{2}(1-\frac{n}{2p})}.\]
	Here, the scale $k\geqslant k(n,p,v,\Lambda,\delta)$ satisfies $\frac{\delta}{\delta^{\prime}}r_{k}\leqslant1$.\\
	\indent Now, we can apply Lemma \ref{lem4.4} to see that the scale $k\geqslant k(n,p,v,\Lambda,\delta)$ such that $|\mathcal{V}_{2r_{k}}(q)-\mathcal{V}_{r_{k}}(q)|\leqslant\delta^{\prime}$ does exist.\ Hence, in the context of the lemma, we get
	\[|\mathcal{V}_{r_{k}}(q)-\ln|\Gamma||<\frac{1}{2}\delta^{\frac{1}{2}(1-\frac{n}{2p})}.\]
	However, if $\mathcal{V}_{1}(q)<\ln N-\delta^{\frac{1}{2}(1-\frac{n}{2p})}$, then by the almost monotonicity formula \eqref{eq2.2}, $\mathcal{V}_{r_{k}}(q)<\ln N-\delta^{\frac{1}{2}(1-\frac{n}{2p})}+C(n,p,v)\delta^{1-\frac{n}{2p}}<\ln N-\frac{1}{2}\delta^{\frac{1}{2}(1-\frac{n}{2p})}$.\\
	\indent Thus, $\ln|\Gamma|<\ln N$ and $|\Gamma|\leqslant N-1$, which completes the proof.
\end{proof}

\subsection{Bubble tree structure}
In this subsection, we will construct the bubble tree structure and prove our diffeomorphism finiteness Theorem \ref{thm1.2}.\\
\indent We have built the neck regions of our bubble tree in Theorem \ref{thm4.5}.\ What is left is to construct the various body regions.\ The body regions are designed to have uniform lower bound on harmonic radius.\ By the proof of Proposition \ref{prop4.1}, we see that if the corresponding group $\Gamma$ is a trivial group, then we do not need to consider the annulus (neck) $A_{\epsilon r,2r}(q)$.\ Indeed, the ball $B(q,2r)$ is GH-close to $B(0^{n},2r)$ and serves as a body region.\ So, we will first assume $|\Gamma|\geqslant2$ and try to apply Lemma \ref{lem4.6}.\ We begin with the following lemma.
\begin{lem}\label{lem4.7}
	For any sufficiently small $\delta\leqslant\delta(n,p,v)$, there exists $r_{0}=r_{0}(n,p,v,\Lambda,\delta)$ and $N=N(n,p,v,\Lambda,\delta)$ with the following property.\ Let $M^{n}$ be a manifold with $\int_{M^{n}}|Rm|^{\frac{n}{2}}\leqslant\Lambda$, $\int_{M^{n}}|Ric|^{p}\leqslant\delta^{2p-n}$ and $vol(B(x,r))\geqslant vr^{n}\ \text{for all}\ x\in M^{n}\ \text{and}\ r\leqslant\delta^{-1}$.\ Then there exist points $\{x_{j}\}_{1}^{m}\subset B(q,1)$ with $m\leqslant N$, and scales $\{k_{j}\}_{1}^{m}\subset\mathbb{N}$ with $r_{j}:=r_{k_{j}}>r_{0}$ such that
	\begin{enumerate}
		\item $T_{k_{j}}^{\delta}(x_{j})=0$;
		\item if $s_{j}$ denotes the largest integer such that $\mathcal{V}_{r_{s_{j}}}(x_{j})\geqslant\ln|\Gamma_{j}|-\delta^{\frac{1}{2}(1-\frac{n}{2p})}$, then for any $x\in B(x_{j},2r_{s_{j}})$, we have $\mathcal{V}_{\frac{1}{2}r_{s_{j}}}(x)<\ln|\Gamma_{j}|-\delta^{\frac{1}{2}(1-\frac{n}{2p})}$;
		\item if $x\in B(q,1)\setminus\bigcup_{j}B(x_{j},r_{j})$, then $r_{h}^{\alpha}(x)>r_{0}$.
	\end{enumerate}
\end{lem} 
\begin{proof}
	Let $\delta\leqslant\delta(n,p,v)$ satisfy Lemma \ref{lem4.4} with $\eta=\eta(n,p,v,\Lambda,\delta)$ to be chosen later.\ For any $x\in M^{n}$, there exists a scale $k_{x}\in[k(n,p,v,\eta),4k(n,p,v,\eta)]$ so that $T_{k_{x}}^{\eta}(x)=0$.\ Pick a covering $\{B(x,r_{k_{x}})\}$ of $B(q,1)$ and choose a finite subcovering $\{B(y_{j},r_{j})\}_{1}^{m}$, where $r_{j}=r_{k_{j}}$, $k_{j}=k_{y_{j}}$ and all the balls $\{B(y_{j},\frac{1}{10}r_{j})\}_{1}^{m}$ are disjoint.\ Thus, $m\leqslant N(n,p,v,\eta)$ by volume doubling property \eqref{eq2.4}.\\
	\indent By Proposition \ref{prop4.1} and the proof of Lemma \ref{lem4.6}, for any sufficiently small $\epsilon=\epsilon(n,p,v,\delta)$, we can choose $\eta(n,p,v,\Lambda,\epsilon,\delta)$ so that for each $x\in A_{\epsilon r_{j},2r_{j}}(y_{j})$, we have $r_{h}^{\alpha}(x)\geqslant r(n,v)\epsilon r_{j}>r_{0}(n,p,v,\epsilon,\eta)$ and $|\mathcal{V}_{\frac{1}{2}\sqrt{\epsilon}r_{j}}(y_{j})-\ln|\Gamma_{j}||<\frac{1}{2}\delta^{\frac{1}{2}(1-\frac{n}{2p})}$.\ Also, $\epsilon$ and $\eta$ can be chosen to satisfy $T_{k_{j}}^{\delta}(x)=0$ for all $x\in B(y_{j},4\sqrt{\epsilon}r_{j})$.\\
	\indent For each $x\in B(y_{j},4\sqrt{\epsilon}r_{j})$, let $s_{j}(x)$ denote the largest integer such that $\mathcal{V}_{r_{s_{j}(x)}}(x)\geqslant\ln|\Gamma_{j}|-\delta^{\frac{1}{2}(1-\frac{n}{2p})}$.\ Here, we need $\delta^{\frac{1}{2}(1-\frac{n}{2p})}\ll\ln2$.\ Consider 
	\[W_{j}:=\{x\in B(y_{j},4\sqrt{\epsilon}r_{j}):B(x,4r_{s_{j}(x)})\subset B(y_{j},4\sqrt{\epsilon}r_{j})\}.\]
	Note that $r_{s_{j}(y_{j})}\leqslant\sqrt{\epsilon}r_{j}$, so $y_{j}\in W_{j}$.\\ 
	\indent Now let $s_{j}:=\max\{s_{j}(x):x\in W_{j}\}$ with $x_{j}$ the corresponding point, i.e., $s_{j}(x_{j})=s_{j}$.\ Since $B(x_{j},4r_{s_{j}})\subset B(y_{j},4\sqrt{\epsilon}r_{j})$, we know that for any $x\in B(x_{j},2r_{s_{j}})$, $\mathcal{V}_{r_{s_{j}+1}}(x)=\mathcal{V}_{\frac{1}{2}r_{s_{j}}}(x)<\ln|\Gamma_{j}|-\delta^{\frac{1}{2}(1-\frac{n}{2p})}$.\\
	\indent Consider the collection of balls $\{B(x_{j},r_{j})\}_{1}^{m}$.\ Conditions (1) and (2) are satisfied by construction.\ If $x\in B(q,1)\setminus\bigcup_{j}B(x_{j},r_{j})$, then since $B(x_{j},\frac{3}{2}r_{j})\supseteq B(y_{j},r_{j})$ and $\{B(y_{j},r_{j})\}_{1}^{m}$ cover $B(q,1)$, we have $x\in A_{r_{j},\frac{3}{2}r_{j}}(x_{j})\subseteq A_{\frac{1}{2}r_{j},2r_{j}}(y_{j})$, which implies $r_{h}^{\alpha}(x)\geqslant r_{0}(n,p,v,\Lambda,\delta)$.\ This completes the proof.
\end{proof}
\indent By Lemma \ref{lem4.7} and Theorem \ref{thm2.12}, the open set $B(q,1)\setminus\overline{\bigcup_{j}B(x_{j},r_{j})}$, which will be referred as a body region, has finite diffeomorphism types.\\
\indent We are now in the position to prove the following bubble tree decomposition result, which is stronger than Theorem \ref{thm1.2}.\ The proof is the same as the one in Cheeger-Naber's paper \cite{cheeger2015regularity}, since we have established all the tools they used in constructing the bubble tree.

\begin{thm}\label{thm4.8}
	 Given $p>\frac{n}{2}$ and positive constants $v,D,\lambda$ and $\Lambda$, if $(M^{n},g)$ is a closed Riemannian manifold satisfying \eqref{eq1.2}, \eqref{eq1.4}, \eqref{eq1.6} and \eqref{eq1.7}, then there exists a decomposition of $M^{n}$
	 \begin{align}\label{eq4.10}
	 	M^{n}=\mathcal{B}_{1}^{0}\cup\bigcup_{i_{1}=1}^{N_{1}}\mathcal{N}_{i_{1}}^{1}\cup\bigcup_{i_{1}=1}^{N_{1}}\mathcal{B}_{i_{1}}^{1}\cup\cdots\cup\bigcup_{i_{m}=1}^{N_{m}}\mathcal{N}_{i_{m}}^{m}\cup\bigcup_{i_{m}=1}^{N_{m}}\mathcal{B}_{i_{m}}^{m}
	 \end{align}
 into open sets with the following properties:
 \begin{enumerate}
 	\item if $x\in\mathcal{B}_{i}^{l}$, then $r_{h}^{\alpha}(x)\geqslant r_{0}(n,p,v,\lambda,\Lambda,D)diam(\mathcal{B}_{i}^{l})$;
 	\item each neck $\mathcal{N}_{i}^{l}$ is diffeomorphic to $(0,1)\times S^{n-1}/\Gamma_{i}^{l}$, for some $\Gamma_{i}^{l}\subset O(n)$ with $|\Gamma_{i}^{l}|\leqslant C(n,v)$;
 	\item $\mathcal{N}_{i}^{l}\cap\mathcal{B}_{i}^{l}$ is diffeomorphic to $(0,1)\times S^{n-1}/\Gamma_{i}^{l}$;
 	\item $\mathcal{B}_{j}^{l-1}\cap\mathcal{N}_{i}^{l}$ are either empty or diffeomorphic to $(0,1)\times S^{n-1}/\Gamma_{i}^{l}$;
 	\item $N_{l}\leqslant N(n,p,v,\lambda,\Lambda,D)$ and $m\leqslant m(n,v)$.
 \end{enumerate}
\end{thm}

\begin{proof}
Without loss of generality, we can assume $\lambda=1$.\ Fix $\epsilon=\epsilon(n,v)$ from Theorem \ref{thm4.5} with $\delta=\delta(n,p,v,\Lambda,D,\epsilon)=\delta(n,p,v,\Lambda,D)$ sufficiently small to satisfy Proposition \ref{prop4.1}, Theorem \ref{thm4.5}, Lemma \ref{lem4.6} and Lemma \ref{lem4.7}.\ Since we want to apply Lemma \ref{lem4.7} on the whole manifold $M^{n}$ instead of $B(q,1)$, we need $\delta$ to be dependent on $D$.\ After rescaling, we may suppose that $(M^{n},g)$ is a closed manifold with $diam(M^{n})\leqslant D\delta^{-1}$, $\int_{M^{n}}|Rm|^{\frac{n}{2}}\leqslant\Lambda$, $\int_{M^{n}}|Ric|^{p}\leqslant\delta^{2p-n}$, $vol(B(x,r))\geqslant vr^{n}\ \text{for all}\ x\in M^{n}$ and $ r\leqslant\delta^{-1}$.\\
\indent To begin with, we apply Lemma \ref{lem4.7} to produce a collection of balls $\{B(x_{i}^{1},r_{i}^{1})\}_{1}^{N_{1}}$ such that $r_{i}^{1}=r_{k_{i}^{1}}>r_{0}(n,p,v,\Lambda,D)D$, $N_{1}\leqslant N(n,p,v,\Lambda,D)$, $T_{k_{i}^{1}}^{\delta}(x_{i}^{1})=0$ and if $x\in M^{n}\setminus\bigcup_{i}B(x_{i}^{1},r_{i}^{1})$, then $r_{h}^{\alpha}(x)>r_{0}(n,p,v,\Lambda,D)D$.\ Moreover, if we denote by $\Gamma_{i}^{1}$ the group associated to $B(x_{i}^{1},r_{1}^{1})$, then if $s_{i}^{1}$ denotes the largest integer such that $\mathcal{V}_{r_{s_{i}^{1}}}(x_{i}^{1})\geqslant\ln|\Gamma_{i}^{1}|-\delta^{\frac{1}{2}(1-\frac{n}{2p})}$, then for any $x\in B(x_{i}^{1},2r_{s_{i}^{1}})$, we have \[\mathcal{V}_{\frac{1}{2}r_{s_{i}^{1}}}(x)<\ln|\Gamma_{i}^{1}|-\delta^{\frac{1}{2}(1-\frac{n}{2p})}.\]
Define $\mathcal{B}_{1}^{0}=M^{n}\setminus\overline{\bigcup_{i}B(x_{i}^{1},r_{i}^{1})}$ as the first body region.\ Then we can write 
\begin{align}\label{eq4.11}
	M^{n}=\mathcal{B}_{1}^{0}\cup\bigcup_{i}B(x_{i}^{1},2r_{i}^{1}),
\end{align}
and $B(x_{i}^{1},2r_{i}^{1})\cap\mathcal{B}_{1}^{0}$ is diffeomorphic to $(0,1)\times S^{n-1}/\Gamma_{i}^{1}$ by Proposition \ref{prop4.1}.\\
\indent Now, let us inductively construct a decomposition of $M^{n}$,
\begin{align}\label{eq4.12}
	M^{n}=\mathcal{B}_{1}^{0}\cup\bigcup_{i_{1}=1}^{N_{1}}\mathcal{N}_{i_{1}}^{1}\cup\bigcup_{i_{1}=1}^{N_{1}}\mathcal{B}_{i_{1}}^{1}\cup\cdots\cup\bigcup_{i_{t}=1}^{N_{t}}\mathcal{N}_{i_{t}}^{t}\cup\bigcup_{i_{t}=1}^{N_{t}}\mathcal{B}_{i_{t}}^{t}\cup\bigcup_{a=1}^{N_{t+1}}B(x_{a},2r_{a}^{t+1})
\end{align}
with the following properties:
\begin{enumerate}
	\item if $x\in\mathcal{B}_{i}^{l}$, then $r_{h}^{\alpha}(x)\geqslant r_{0}(n,p,v,\Lambda,D)diam(\mathcal{B}_{i}^{l})$;
	\item each neck $\mathcal{N}_{i}^{l}$ is diffeomorphic to $(0,1)\times S^{n-1}/\Gamma_{i}^{l}$, for some $\Gamma_{i}^{l}\subset O(n)$ with $|\Gamma_{i}^{l}|\leqslant C(n,v)$;
	\item $\mathcal{N}_{i}^{l}\cap\mathcal{B}_{i}^{l}$ is diffeomorphic to $(0,1)\times S^{n-1}/\Gamma_{i}^{l}$;
	\item $\mathcal{B}_{j}^{l-1}\cap\mathcal{N}_{i}^{l}$ are either empty or diffeomorphic to $(0,1)\times S^{n-1}/\Gamma_{i}^{l}$;
	\item $N_{l}\leqslant N(n,p,v,\Lambda,D)$;
	\item if $\mathcal{B}_{j}^{l-1}\cap\mathcal{N}_{i}^{l}\neq\emptyset$, then $|\Gamma_{i}^{l}|\leqslant|\Gamma_{j}^{l-1}|-1$;
	\item $r_{a}^{t+1}=r_{k_{a}^{t+1}}$ with $T_{k_{a}^{t+1}}^{\delta}(x_{a})=0$, and $\mathcal{B}_{i}^{t}\cap B(x_{a},2r_{a}^{t+1})\subseteq A_{\frac{1}{2}r_{a}^{t+1},2r_{a}^{t+1}}(x_{a})$;
	\item if $s_{a}^{t+1}$ is the largest integer such that $\mathcal{V}_{r_{s_{a}^{t+1}}}(x_{a})\geqslant\ln|\Gamma_{a}^{t+1}|-\delta^{\frac{1}{2}(1-\frac{n}{2p})}$, then for any $x\in B(x_{a},2r_{s_{a}^{t+1}})$, we have $\mathcal{V}_{\frac{1}{2}r_{s_{a}^{t+1}}}(x)<\ln|\Gamma_{a}^{t+1}|-\delta^{\frac{1}{2}(1-\frac{n}{2p})}$.
\end{enumerate}
Note that \eqref{eq4.11} provides the case $t=0$.\ So let us suppose that the decomposition \eqref{eq4.12} has been built for $t\geqslant0$, and we are going to construct for $t+1$.\\
\indent We first use Theorem \ref{thm4.5} to see that there exists an open set
\[A_{\frac{1}{2}r_{s_{a}^{t+1}},2r_{k_{a}^{t+1}}}(x_{a})\subset\mathcal{N}_{a}^{t+1}\subset A_{\frac{1}{2}(1-\epsilon)r_{s_{a}^{t+1}},2(1+\epsilon)r_{k_{a}^{t+1}}}(x_{a})\]
and a diffeomorphism $\phi_{a}^{t+1}:\mathcal{N}_{a}^{t+1}\to A_{\frac{1}{2}r_{s_{a}^{t+1}},2r_{k_{a}^{t+1}}}(0)$ where $0\in\mathbb{R}^{n}/\Gamma_{a}^{t+1}$.\ By Lemma \ref{lem4.6} and condition (8), there exists a radius $r_{a}=r(n,p,v,\Lambda,\delta)r_{s_{a}^{t+1}}=r(n,p,v,\Lambda,D)r_{s_{a}^{t+1}}<1$ such that for any $x\in B(x_{a},2r_{s_{a}^{t+1}})$,
\[\mathcal{V}_{r_{a}}(x)<\ln(|\Gamma_{a}^{t+1}|-1)+\delta^{\frac{1}{2}(1-\frac{n}{2p})}.\]
\indent Now we apply Lemma \ref{lem4.7} to each ball $B(x_{a},r_{s_{a}^{t+1}})$ to construct a collection of balls $\{B(x_{aj},r_{aj}^{t+2})\}$ with $r_{aj}^{t+2}=r_{k_{aj}^{t+2}}$ and $T_{k_{aj}^{t+2}}^{\delta}(x_{aj})=0$.\ There are at most $N(n,p,v,\Lambda,D)$ such balls in total and for any $x\in B(x_{a},r_{s_{a}^{t+1}})\setminus\bigcup_{j}B(x_{aj},r_{aj}^{t+2})$, $r_{h}^{\alpha}(x)>r_{0}(n,p,v,\Lambda,D)r_{s_{a}^{t+1}}$.\ Also, from the proof of Lemma \ref{lem4.7}, we can assume that $r_{aj}^{t+2}<r(n,p,v,\Lambda,D)r_{s_{a}^{t+1}}=r_{a}$ and $|\mathcal{V}_{r_{aj}^{t+2}}(x_{aj})-\ln|\Gamma_{aj}^{t+2}||<\delta^{\frac{1}{2}(1-\frac{n}{2p})}$ where $\Gamma_{aj}^{t+2}$ is the corresponding group of $B(x_{aj},r_{aj}^{t+2})$.\\
\indent On the other hand, by the almost monotonicity formula \eqref{eq2.2}, we have
\[\mathcal{V}_{r_{aj}^{t+2}}(x_{aj})<\mathcal{V}_{r_{a}}(x_{aj})+C(n,p,v)\delta^{1-\frac{n}{2p}}<\ln(|\Gamma_{a}^{t+1}|-1)+2\delta^{\frac{1}{2}(1-\frac{n}{2p})},\]
which implies $|\Gamma_{aj}^{t+2}|\leqslant|\Gamma_{a}^{t+1}|-1$.\ If we put
\[\mathcal{B}_{a}^{t+1}=B(x_{a},r_{s_{a}^{t+1}})\setminus\overline{\bigcup_{j}B(x_{aj},r_{aj}^{t+2})},\]
and write
\[M^{n}=\mathcal{B}_{1}^{0}\cup\cdots\cup\bigcup_{i_{t}=1}^{N_{t}}\mathcal{N}_{i_{t}}^{t}\cup\bigcup_{i_{t}=1}^{N_{t}}\mathcal{B}_{i_{t}}^{t}\cup\bigcup_{a=1}^{N_{t+1}}\mathcal{N}_{a}^{t+1}\cup\bigcup_{a=1}^{N_{t+1}}\mathcal{B}_{a}^{t+1}\cup\bigcup B(x_{aj},2r_{aj}^{t+2}),\]
then one can check the inductive conditions below \eqref{eq4.12} easily by shrinking neck regions if necessary.\\
\indent Once we have the inductive decomposition, we could finish the proof immediately.\ In fact, by condition (6) and the upper bound $|\Gamma_{i}^{l}|\leqslant C(n,v)$, we see that the decomposition will stop within $m$ steps where $m\leqslant m(n,v)$.\ This completes the proof.
\end{proof}
Once we have the bubble tree decomposition, our diffeomorphism finiteness statement of Theorem \ref{thm1.2} can be proved directly.\ Notice that all body regions have finite diffeomorphism types by Theorem \ref{thm2.12} and any neck region is diffeomorphic to $(0,1)\times S^{n-1}/\Gamma$ for some $\Gamma\subseteq O(n)$ with $|\Gamma|<C(n,v)$.\ See \cite{cheeger2015regularity} for more details.\ Hence, we complete the proof of Theorem \ref{thm1.2}.\\

\noindent \textbf{Data Availability }This paper has no associated data.

	\nocite{*}
	\bibliographystyle{siam}
	\bibliography{refe}

\noindent E-mail: xqian22@m.fudan.edu.cn\\[2pt]
\emph{School of Mathematical Sciences, Fudan University, Shanghai
	200433, People's Republic of China}
\end{document}